\documentclass[12pt,reqno]{amsart}
\usepackage[margin=1in]{geometry}
\usepackage{dsfont}
\usepackage{amsmath,amssymb,amsthm,graphicx,amsxtra, setspace}
\usepackage[utf8]{inputenc}
\usepackage{mathrsfs}
\usepackage{alltt}
\usepackage{relsize}
\usepackage{hyperref}
\usepackage{aliascnt}
\usepackage{tikz}
\usepackage{mathtools}
\usepackage{multicol}
\usepackage{upgreek}
\usepackage{graphicx,type1cm,eso-pic,color}
\allowdisplaybreaks

\newtheorem{theorem}{Theorem}[section]

\newaliascnt{lemma}{theorem}
\newtheorem{lemma}[lemma]{Lemma}
\aliascntresetthe{lemma} 

\newaliascnt{proposition}{theorem}
\newtheorem{proposition}[proposition]{Proposition}
\aliascntresetthe{proposition}

\newaliascnt{assumption}{theorem}

\aliascntresetthe{assumption}

\newaliascnt{corollary}{theorem}
\newtheorem{corollary}[corollary]{Corollary}
\aliascntresetthe{corollary}

\newaliascnt{definition}{theorem}
\newtheorem{definition}[definition]{Definition}
\aliascntresetthe{definition}

\newaliascnt{example}{theorem}

\aliascntresetthe{example}

\newaliascnt{remark}{theorem}
\newtheorem{remark}[remark]{Remark}
\aliascntresetthe{remark}

\newaliascnt{hypothesis}{theorem}
\newtheorem{hypothesis}[hypothesis]{Hypothesis}
\aliascntresetthe{hypothesis}

\newaliascnt{property}{theorem}

\aliascntresetthe{property}

\let\originalleft\left
\let\originalright\right
\renewcommand{\left}{\mathopen{}\mathclose\bgroup\originalleft}
\renewcommand{\right}{\aftergroup\egroup\originalright}

\newcommand{\Tr}{\mathop{\mathrm{Tr}}}
\renewcommand{\d}{\/\mathrm{d}\/}

\def\e{\varepsilon}

\def\t{t\wedge\tau_N^n}
\def\s{t\wedge\tau_N}

\def\T{T\wedge\tau_N^n}
\def\L{\mathrm{L}}

\def\I{\mathrm{I}}
\def\F{\mathrm{F}}
\def\C{\mathrm{C}}
\def\f{\mathbf{f}}
\def\J{\mathrm{J}}
\def\B{\mathrm{B}}
\def\D{\mathrm{D}}

\def\E{\mathbb{E}}
\def\X{\mathrm{X}}

\def\v{\mathbf{v}}
\def\V{\mathbb{V}}

\def\W{\mathrm{W}}
\def\G{\mathcal{O}}

\def\V{\mathbb{V}}
\def\wi{\widetilde}

\def\U{\mathrm{U}}
\def\P{\mathbb{P}}
\def\u{\mathbf{u}}
\def\H{\mathrm{H}}

\renewcommand{\d}{\/\mathrm{d}\/}

\newcommand{\Addresses}{{
		\footnote{
			
			\noindent \textsuperscript{1}Department of Mathematics, Indian Institute of Technology Roorkee-IIT Roorkee,	Haridwar Highway, Roorkee, Uttarakhand 247667, INDIA.\par\nopagebreak	\noindent  \textit{e-mail:} \texttt{manilfma@iitr.ac.in, maniltmohan@gmail.com.}
			
		\noindent \textsuperscript{*}Corresponding author.

			\textit{Key words: Stochastic Burgers-Huxley equation, strong solution, large deviation, exit time estimate, invariant measure, ergodicity.} 
			
			Mathematics Subject Classification (2010): 65C30, 37L40, 76D03. 
			
		}}}

\begin{document}
	
	\title[Stochastic Burgers-Huxley Equations]{Stochastic Burgers-Huxley Equations: Global Solvability, Large Deviations and Ergodicity\Addresses}
	\author[M. T. Mohan ]{Manil T. Mohan\textsuperscript{1*}}

	\maketitle
	
	\begin{abstract}
In this work, we consider the stochastic Burgers-Huxley equation perturbed by multiplicative Gaussian noise, and discuss about the global solvability results and asymptotic behavior of solutions. We show the existence of a global strong solution of the stochastic Burgers-Huxley equation, by making use of a  local monotonicity property of the linear and nonlinear operators and  a stochastic generalization of  localized version of the Minty-Browder technique. We then discuss about the inviscid limit of the stochastic Burgers-Huxley equation to Burgers as well as Huxley equations. By considering the noise to be additive Gaussian, Exponential estimates for exit from a ball of radius $R$ by time $T$ for solutions of the stochastic Burgers-Huxley equation is derived, and then studied in the context of Freidlin-Wentzell type large deviations principle. Finally, we establish the existence of a unique ergodic and strongly mixing invariant measure for the stochastic Burgers-Huxley equation with additive Gaussian noise, using the exponential stability of solutions.
	\end{abstract}

	\section{Introduction}\label{sec1}\setcounter{equation}{0}  We consider the \emph{ stochastic Burgers-Huxley equation} for $(x,t)\in\mathcal{O}\times(0,T)=(0,1)\times(0,T)$ with a random force  as (see \cite{JS})
	\begin{align}\label{1.1}
du(t)=\left(\nu\frac{\partial^2u(t)}{\partial x^2}-\alpha u(t)\frac{\partial u(t)}{\partial x}+\beta u(t)(1-u(t))(u(t)-\gamma)\right)dt+\sigma(t,u(t))dW(t),
	\end{align}
	where $\alpha>0$ is the advection coefficient, $\nu,\beta>0$ and $\gamma\in(0,1)$ are parameters. In \eqref{1.1}, $W(\cdot)$ is an $\L^2(\mathcal{O})$-valued $Q$-Wiener process. 
	 We supplement \eqref{1.1} with Dirichlet boundary conditions 
	\begin{align}\label{1.5}
	u(0,t)=u(1,t)=0,
	\end{align}
	and the initial condition
	\begin{align}\label{1.6}
	u(x,0)=u_0(x), \ x\in\overline{\mathcal{O}}.
	\end{align}
	 Equation \eqref{1.1} describes a prototype model for describing the interaction between reaction mechanisms, convection effects and diffusion transports.  	Our goal in this work is to study the global solvability results as well as asymptotic analysis of solutions to the problem \eqref{1.1} with boundary and initial conditions \eqref{1.5} and \eqref{1.6}. We use a stochastic generalization of localized version of the Minty-Browder technique to obtain the global strong solution. A local monotonicity property of the linear and nonlinear operators is exploited in the proofs. The inviscid limit of the stochastic Burgers-Huxley equation to Burgers and Huxley equations is also discussed. For additive Gaussian noise case, using energy estimates and Doob's martingale inequality, exponential estimates for exit time for solutions of the stochastic Burgers-Huxley equation is obtained. We also established a similar estimate for exit time by using Freidlin-Wentzell type large deviations principle. The existence of a unique ergodic and strongly mixing invariant measure for the stochastic Burgers-Huxley equation is established by  making use of the exponential stability of solutions.

	 For $\alpha=0$, the equation \eqref{1.1} takes the form 
		\begin{align}\label{2}
	du(t)=\left(\nu\frac{\partial^2u(t)}{\partial x^2}+\beta u(t)(1-u(t))(u(t)-\gamma)\right)dt+\sigma(t,u(t)) dW(t),
	\end{align}
which is known as the \emph{stochastic Huxley equation} and it describes nerve pulse propagation in nerve fibers and wall motion in liquid crystals (\cite{XYW1}). For $\beta=0$ and $\alpha=1$, the equation \eqref{1.1} can be reduced to
	\begin{align}\label{3}
du(t)=\left(\nu\frac{\partial^2u(t)}{\partial x^2}- u(t)\frac{\partial u(t)}{\partial x}\right)dt+\sigma(t,u(t)) dW(t),
\end{align}
which is the well-known \emph{stochastic viscous Burgers equation}.  In \cite{JMB}, Burgers studied the deterministic model for modeling the turbulence phenomena (see \cite{HB,JMB1} also). The authors in \cite{GDP} proved the existence and uniqueness of a global mild solution as well as the existence of an invariant measure of the stochastic Burgers' equation perturbed by cylindrical Gaussian noise. The authors in \cite{BCJ94} studied the  Burgers Equation perturbed by a white noise in space and time, and proved the existence of solutions by showing that the Cole-Hopf transformation is meaningful also in the stochastic case. The global existence and uniqueness of the strong, weak and mild solutions for one-dimensional Burgers equation perturbed by L\'evy noise is established in \cite{ZDTG}. The Burgers equation perturbed by a multiplicative white noise is considered in \cite{GDP2} and proved the existence and uniqueness of the global solution as well as the strong Feller property and irreducibility for the corresponding transition semigroup (see also Chapter 14, \cite{GDJZ}). Moreover, the existence and uniqueness of an invariant measure is also established. Exponential ergodicity for stochastic Burgers equations is established in \cite{BGBM}. Control problems and dynamic programming of the stochastic Burgers equation have been carried out in \cite{GDP1,GDP2,HCR}, etc.  The stochastic generalized Burgers-Huxley equation perturbed by cylindrical Gaussian noise is considered in \cite{MTM2} and proved  the existence of a unique global mild solution using a fixed point method and stopping time arguments. The works \cite{AAIG,DBAJ,MHJV,ATHZ}, etc considered numerical analysis of stochastic Burgers equations. Different mathematical problems regarding stochastic Burgers equation is available in the literature and interested readers are referred to see \cite{RBLB,ZBLD,LBGG,PGMJ,FYW,ATHZ1}, etc and references therein.

The rest of the paper is organized as follows. In the next section, we provide the abstract formulation of the problem and provide the necessary function spaces needed to obtain the global solvability results of the system \eqref{1.1}-\eqref{1.6}. A local monotonicity property as well as hemicontinuity of the linear and nonlinear operators in $\L^{\infty}$  ball is proved in the same section (Theorem \ref{monotone} and Lemma \ref{lem2.5}). The existence and uniqueness of global strong solution is established in section \ref{sec3} using a stochastic generalization of localized version of the Minty-Browder technique (Theorem \ref{exis}).  The inviscid limit of the stochastic Burgers-Huxley equation to the stochastic Burgers equation (as $\beta\to 0$, Proposition \ref{prop5.1}) as well as to  the stochastic Huxley equation (as $\alpha\to 0$, Proposition \ref{prop4.2}) is discussed in section \ref{sec5}. The system \eqref{1.1}-\eqref{1.6} perturbed by additive Gaussian noise is considered in sections \eqref{sec6} and \eqref{sec9}. Using energy estimates and Doob's martingale inequality, the exponential estimates for exit from a ball of radius $R$ by time $T$ for strong solutions of the stochastic Burgers-Huxley equation is derived in section \ref{sec6} (Remark \ref{rem5.10}). We studied the exit time estimates in the context of Freidlin-Wentzell type large deviations principle also in the same section (Theorem \ref{thm4.8}). In section \ref{sec9}, we prove the exponential moment energy estimates for the strong solution to the system \eqref{1.1}-\eqref{1.6} and also  the exponential stability of solutions (Theorems \ref{expe} and \ref{exps}).  Finally, we establish the existence of a unique ergodic and strongly mixing invariant measure for the stochastic Burgers-Huxley equation, using the exponential stability of solutions (Theorems \ref{EIM} and \ref{UEIM}).

\section{Mathematical formulation}\label{sec2}\setcounter{equation}{0}
In this section, we present the necessary function spaces, properties of linear and nonlinear operators used to obtain the global solvability results of the system \eqref{1.1}-\eqref{1.6}. We show that the sum of linear and nonlinear operator is locally monotone and  hemicontinuous. 
\subsection{Functional setting}
Let $\C_0^{\infty}(\mathcal{O})$ denotes the space of all infinitely differentiable functions with compact support on $\mathcal{O}$. For $p\in[2,\infty)$, the Lebesgue spaces are denoted by $\L^p(\mathcal{O})$ and the Sobolev spaces are denoted by $\H^{k}(\mathcal{O})$. The norm in $\L^2(\mathcal{O})$ is denoted by $\|\cdot\|_{\L^2}$ and  the inner product in $\L^2(\mathcal{O})$ is denoted by $(\cdot,\cdot)$.  Let $\H_0^1(\mathcal{O})$ denotes closure of $\C_0^{\infty}(\mathcal{O})$ in $\|\partial_x\cdot\|_{\L^2}$ norm. As $\mathcal{O}$ is a bounded domain, note that $\|\partial_x\cdot\|_{\L^2}$  defines a norm on $\H^1_0(\mathcal{O})$ and we have the continuous embedding $\H_0^1(\mathcal{O})\subset\L^2(\mathcal{O})\subset\H^{-1}(\mathcal{O})$, where $\H^{-1}(\mathcal{O})$ is the dual space of $\H_0^1(\mathcal{O})$. Remember that the embedding of $\H_0^1(\mathcal{O})\subset\L^2(\mathcal{O})$ is compact. The duality paring between $\H_0^1(\mathcal{O})$ and $\H^{-1}(\mathcal{O})$ is denoted by $\langle\cdot,\cdot\rangle$. In one dimensions, we have the following continuous embedding: $\H_0^1(\mathcal{O})\subset\L^{\infty}(\mathcal{O})\subset\L^p(\mathcal{O}),$ for all $p\in[1,\infty)$ (see \cite{HM}).
\subsection{Linear operator}
Let $A$ denotes the self-adjoint and unbounded operator on $\L^2(\mathcal{O})$ defined by\footnote{Strictly speaking one has to define $A:=-\frac{d^2}{dx^2}$.}
\begin{align*}
A u:=-\frac{\partial^2u}{\partial x^2},
\end{align*}
with domain $D(A)= \H^2(\mathcal{O})\cap\H_0^1(\mathcal{O})=\{u\in\H^2(\mathcal{O}):u(0)=u(1)=0\}.$ The eigenvalues and the corresponding eigenfunctions of $A$ are given by
\begin{align*}
\lambda_k=k^2 \pi^2, \text{and } \  w_k(x)=\sqrt{\frac{2}{\pi}}\sin (k\pi x), \ k=1,2\ldots.
\end{align*}
Since $\mathcal{O}$ is a bounded domain, $A^{-1}$ exists and is a compact operator on $\L^2(\mathcal{O})$. Moreover, one can define the fractional powers of $A$ and 
$$\|A^{1/2}u\|_{\L^2}^2=\sum_{j=1}^{\infty}\lambda_j\langle u,w_j\rangle|^2\geq \lambda_1\sum_{j=1}^{\infty}\langle u,w_j\rangle|^2=\lambda_1\|u\|_{\L^2}^2=\pi^2\|u\|_{\L^2}^2,$$ which is the Poincar\'e inequality. Note also that $\|u\|_{\H^{s}}=\|A^{s/2}u\|_{\L^2},$ for all $s\in\mathbb{R}$. An integration by parts yields $$(Au,v)=(\partial_xu,\partial_xv)=:a(u,v), \ \text{ for all } \ v\in\H_0^1(\mathcal{O}),$$ so that $A:\H_0^1(\mathcal{O})\to\H^{-1}(\mathcal{O})$. 
\subsection{Nonlinear operators}
Let us now define $b:\H_0^1(\mathcal{O})\times\H_0^1(\mathcal{O})\times \H_0^1(\mathcal{O})\to\mathbb{R}$ as $$b(u,v,w)=\int_0^1u(x)\frac{\partial v(x)}{\partial x}w(x)dx.$$ 
Using an integration by parts and boundary conditions, it can be easily seen that 
\begin{align}\label{6}
b(u,u,u)=(u\partial_xu,u)=\int_0^1(u(x))\frac{\partial u(x)}{\partial x}u(x)d x=\frac{1}{3}\int_0^1\frac{\partial}{\partial x}(u(x))^{3}d x=0.
\end{align}
For $w\in\L^2(\mathcal{O})$, we can define an operator $B(\cdot,\cdot):\H_0^1(\mathcal{O})\times\H_0^1(\mathcal{O})\to\L^2(\mathcal{O})$ by 
\begin{align*}
(B(u,v),w)=b(u,v,w)\leq\|u\|_{\L^{\infty}}\|\partial_xv\|_{\L^2}\|w\|_{\L^{2}}\leq\|u\|_{\H_0^1}\|v\|_{\H_0^1}\|w\|_{\L^2},
\end{align*}
so that $\|B(u,v)\|_{\L^2}\leq \|u\|_{\H_0^1}\|v\|_{\H_0^1}$. We denote $B(u)=B(u,u)$, so that one can easily obtain $\|B(u)\|_{\L^2}\leq \|u\|_{\H_0^1}^{2}$. Moreover, for all $v\in\H_0^1(\mathcal{O})$, we have 
\begin{align*}
|\langle B(u),v\rangle| =|\langle u\partial_xu,v\rangle| =\left|-\frac{1}{2}\langle u^2,\partial_xv\rangle\right| \leq \frac{1}{2}\|u\|_{\L^4}^2\|\partial_xv\|_{\L^2},
\end{align*}
so that $B(\cdot):\L^4(\mathcal{O})\to\H^{-1}(\mathcal{O})$ and 
\begin{align}\label{2p5}
\|B(u)\|_{\H^{-1}}\leq \frac{1}{2}\|u\|_{\L^4}^2. 
\end{align}
Using H\"older's inequality, we have 
\begin{align}\label{2.1}
\|B(u)-B(v)\|_{\L^2}&=\|u\partial_xu-v\partial_xv\|_{\L^2}\leq\|(u-v)\partial_xu\|_{\L^2}+\|v\partial_x(u-v)\|_{\L^2}\nonumber\\&\leq\|u-v\|_{\L^{\infty}}\|u\|_{\H_0^1}+\|v\|_{\L^{\infty}}\|u-v\|_{\H_0^1}\nonumber\\&\leq C(\|u\|_{\H_0^1}+\|v\|_{\H_0^1})\|u-v\|_{\H_0^1},
\end{align}
and hence the operator $B:\H_0^1(\mathcal{O})\to\L^2(\mathcal{O})$ is locally Lipschitz. Moreover, we have 
\begin{align}
\langle B(u)-B(v),w\rangle=\frac{1}{2}(\partial_xu^2-\partial_xv^2,w)=-\frac{1}{2}((u-v)(u+v),\partial_xw).
\end{align}
Using H\"older's inequality, we get 
\begin{align}
|\langle B(u)-B(v),w\rangle|\leq\frac{1}{2}(\|u\|_{\L^4}+\|v\|_{\L^4})\|u-v\|_{\L^4}\|w\|_{\H_0^1},
\end{align}
and hence $\|B(u)-B(v)\|_{\H^{-1}}\leq \frac{1}{2}(\|u\|_{\L^4}+\|v\|_{\L^4})\|u-v\|_{\L^4}$, so that the operator $B:\L^4(\mathcal{O})\to\H^{-1}(\mathcal{O})$ is locally Lipschitz.

Let us define $c(u)=u(1-u)(u-\gamma)$. Using H\"older's and Young's inequalities, we have  
\begin{align}\label{7}
(c(u),u)&=(u(1-u)(u-\gamma),u)=((1+\gamma)u^{2}-\gamma u-u^{3},u)\nonumber\\&=(1+\gamma)(u^{2},u)-\gamma\|u\|_{\L^2}^2-\|u\|_{\L^{4}}^{4}\nonumber\\&\leq(1+\gamma)\|u\|_{\L^4}^2\|u\|_{\L^2}-\gamma\|u\|_{\L^2}^2-\|u\|_{\L^{4}}^{4}\nonumber\\&\leq -\frac{1}{2}\|u\|_{\L^4}^4+\frac{(1+\gamma^2)}{2}\|u\|_{\L^2}^2,
\end{align}
for all $u\in\L^{4}(\mathcal{O})$. Using  H\"older's inequality, for $u,v\in\H_0^1(\mathcal{O})$, we get 
\begin{align}\label{2p7}
\|c(u)-c(v)\|_{\L^2}&=\|(1+\gamma)(u^{2}-v^{2})-\gamma(u-v)-(u^{3}-v^{3})\|_{\L^2}\nonumber\\&\leq (1+\gamma)\|u+v\|_{\L^{\infty}}\|w\|_{\L^2}+\gamma\|w\|_{\L^2}+\|u^2+uv+v^2\|_{\L^{\infty}}\|w\|_{\L^2}\nonumber\\&\leq\left[(1+\gamma)(\|u\|_{\L^{\infty}}+\|v\|_{\L^{\infty}})+\gamma+\frac{1}{2}(\|u\|_{\L^{\infty}}^2+\|v\|_{\L^{\infty}}^2)\right]\|w\|_{\L^2},
\end{align}
and the operator $c(\cdot):\H_0^1(\mathcal{O})\to\L^2(\mathcal{O})$ is locally Lipschitz. For more details see \cite{MTM1}. 

\subsection{Local monotonicity} We show that the operator $F(\cdot)=\nu A+\alpha B(\cdot)-\beta c(\cdot)$ is locally monotone. 
\begin{definition}
	Let $\X$ be a Banach space and let $\X^{'}$ be its topological dual.
	An operator $\mathrm{F}:\mathrm{D}\rightarrow
	\X^{'},$ $\mathrm{D}=\mathrm{D}(\mathrm{F})\subset \X$ is said to be
	\emph{monotone} if
	$\langle\mathrm{F}(x)-\mathrm{F}(y),x-y\rangle\geq
	0,$ for all $x,y\in \mathrm{D}$. 
	The operator $\F(\cdot)$ is said to be \emph{hemicontinuous} if, for all $x, y\in\X$ and $w\in\X'$ $$\lim_{\lambda\to 0}\langle\F(x+\lambda y),w\rangle=\langle\F(x),w\rangle.$$
	The operator $\F(\cdot)$ is called \emph{demicontinuous} if for all $x\in\mathrm{D}$ and $y\in\X$, the functional $x \mapsto\langle \mathrm{F}(x), y\rangle$  is continuous, or in other words, $x_k\to x$ in $\X$ implies $\mathrm{F}(x_k)\xrightarrow{w}\mathrm{F}(x)$ in $\X'$. Clearly demicontinuity implies hemicontinuity. 
\end{definition}
\begin{theorem}\label{monotone}
	For a given $r > 0$, we consider the following (closed) $\L^{\infty}$-ball $\B_r$ in the space $\H_0^1(\mathcal{O})$: $$\B_r := \big\{v \in \H_0^1(\mathcal{O}):\|v\|_{\L^{\infty}}\leq r\big\},$$ then, for any $u\in\H_0^1(\mathcal{O})$ and $v\in\B_r$, we have
	\begin{align}\label{3.7}
&\langle F(u)-F(v),u-v\rangle +\left\{\frac{\alpha^2}{2\nu}r^2+\beta(1+\gamma+\gamma^2)\right\}\|u-v\|_{\L^2}^2\geq \frac{\nu}{2}\|\partial_x(u-v)\|_{\L^2}^2\geq 0.
	\end{align}
\end{theorem}
\begin{proof}
	Let us first consider $\langle F(u)-F(v),u-v\rangle$ and simplify it as 
\begin{align}\label{2.9}
\langle F(u)-F(v),u-v\rangle &=\langle\nu Au+\alpha B(u)+\beta c(u)-(\nu Av+\alpha B(v)-\beta c(v)),u-v\rangle \nonumber\\&=\nu\|\partial_x(u-v)\|_{\L^2}^2+\alpha(B(u)-B(v),u-v)-\beta(c(u)-c(v),u-v).
\end{align}
Using an integration by parts, \eqref{6}, H\"older's  and Young's inequalities, we estimate $|\alpha(B(u)-B(v),u-v)|$ as 
\begin{align}\label{2.10}
|\alpha(B(u)-B(v),u-v)|&=\alpha|(u\partial_xu-v\partial_xv,u-v)|\nonumber\\&=\alpha|((u-v)\partial_xu,u-v)+(v\partial_x(u-v),u-v)|\nonumber\\&=\alpha\left|((u-v)\partial_xv,(u-v))+\frac{1}{2}(v,\partial_x(u-v)^2)\right|\nonumber\\&=\alpha\left|-\frac{1}{2}(v,\partial_x(u-v)^2)\right|=\alpha|-(v,(u-v)\partial_x(u-v))|\nonumber\\&\leq\alpha\|v\|_{\L^{\infty}}\|v-v\|_{\L^2}\|\partial_x(u-v)\|_{\L^2}\nonumber\\&\leq\frac{\nu}{2}\|\partial_x(u-v)\|_{\L^2}^2+\frac{\alpha^2}{2\nu}\|v\|_{\L^{\infty}}^2\|v-v\|_{\L^2}^2. 
\end{align}
Let us now estimate $\beta(c(u)-c(v),u-v)$ as 
\begin{align}\label{2.11}
&\beta(c(u)-c(v),u-v)\nonumber\\&=\beta((1+\gamma)(u^2-v^2)-\gamma(u-v)-(u^3-v^3),u-v)\nonumber\\&=\beta(1+\gamma)((u+v)(u-v),u-v)-\beta\gamma\|u-v\|_{\L^2}^2-\beta((u^2+uv+v^2)(u-v),u-v)\nonumber\\&=\beta(1+\gamma)((u+v)(u-v),u-v)-\beta\gamma\|u-v\|_{\L^2}^2-\beta\|u(u-v)\|_{\L^2}^2-\beta\|v(u-v)\|_{\L^2}^2\nonumber\\&\quad-\beta(uv(u-v),u-v)\nonumber\\&\leq\beta(1+\gamma)\left(\|u(u-v)\|_{\L^2}+\|v(u-v)\|_{\L^2}\right)\|u-v\|_{\L^2}\nonumber\\&\quad-\beta\gamma\|u-v\|_{\L^2}^2-\beta\|u(u-v)\|_{\L^2}^2-\beta\|v(u-v)\|_{\L^2}^2+\frac{\beta}{2}((u^2+v^2)(u-v),u-v)\nonumber\\&\leq {\beta(1+\gamma+\gamma^2)}\|u-v\|_{\L^2}^2,
\end{align}
where we used H\"older's and Young's inequalities. 
Combining \eqref{2.10}-\eqref{2.11} and then substituting it in \eqref{2.9}, we find 
\begin{align}\label{2p13}
&\langle F(u)-F(v),u-v\rangle+\left\{\frac{\alpha^2}{2\nu}\|v\|_{\L^{\infty}}^2+\beta(1+\gamma+\gamma^2)\right\}\|u-v\|_{\L^2}^2 \geq \frac{\nu}{2}\|\partial_x(u-v)\|_{\L^2}^2\geq 0.
\end{align}
 Using the fact that $v\in\B_r$, we get the required result \eqref{3.7}. 
\end{proof}

\begin{corollary}\label{mon1}
	For any $u,v\in\mathrm{L}^2(0, T ; \H_0^1(\mathcal{O}))$ and continuous function $\rho(t)$ on  $t\in(0,T)$, we have
	\begin{align}\label{3.11y}
	&\int_0^Te^{-\rho(t)}\langle F(u(t))-F(v(t)),u(t)-v(t) \rangle d t\nonumber\\&\quad +\int_0^Te^{-\rho(t)}\left\{\frac{\alpha^2}{2\nu}\|v(t)\|_{\L^{\infty}}^2+\beta(1+\gamma+\gamma^2)+\frac{L}{2}\right\}\|u(t)-v(t)\|_{\L^2}^2d t\nonumber\\&\geq \frac{1}{2}\int_0^Te^{-\rho(t)}\|\sigma(t, u) - \sigma(t,
	v)\|^2_{\mathcal{L}_{Q}}d t,
	\end{align}
	where we used Hypothesis \ref{hyp} (H.3). 
\end{corollary}

\begin{lemma}\label{lem2.5}
	The operator $F(\cdot)=\nu A+\alpha B(\cdot)-\beta c(\cdot)$ is demicontinuous. 
\end{lemma}
\begin{proof}
	Let us take a sequence $u^n\to u$ in $\H_0^1(\mathcal{O})$. For $v\in\H_0^1(\mathcal{O})$, we consider 
	\begin{align}\label{214}
	\langle F(u^n)-F(u),v\rangle &=\nu\langle A(u^n)-A(u),v\rangle+\alpha(B(u^n)-B(u),v)-\beta(c(u^n)-c(u),v)\nonumber\\&=\nu(\partial_x(u^n-u),\partial_xv)+\alpha((u^n-u)(u^n+u),\partial_xv)\nonumber\\&\quad-\beta((u^n-u)\left[(u^n+u)(1+\gamma)-(\gamma+(u^n)^2+u^nu+u^2)\right],v)\nonumber\\&\leq\nu\|\partial_x(u^n-u)\|_{\L^2}\|\partial_xv\|_{\L^2}+\alpha\|u^n-u\|_{\L^4}\|u^n+u\|_{\L^4}\|\partial_xv\|_{\L^2}\nonumber\\&\quad+\beta(1+\gamma)\|u^n-u\|_{\L^2}\|u^n+u\|_{\L^2}\|v\|_{\L^{\infty}}+\beta\gamma\|u^n-u\|_{\L^2}\|v\|_{\L^2}\nonumber\\&\quad+\beta\|u^n-u\|_{\L^2}(\|u^n\|_{\L^4}^2+\|u^n\|_{\L^4}\|u\|_{\L^4}+\|u\|_{\L^4}^2)\|v\|_{\L^{\infty}}\nonumber\\&\leq \Big(\nu+\alpha(\|u^n\|_{\L^4}+\|u\|_{\L^4})+C\beta(1+\gamma)(\|u^n\|_{\L^2}+\|u\|_{\L^2})\nonumber\\&\quad+C\beta(\|u^n\|_{\L^4}^2+\|u^n\|_{\L^4}\|u\|_{\L^4}+\|u\|_{\L^4}^2)\Big)\|u^n-u\|_{\H_0^1}\|v\|_{\H^1_0},
	\end{align} 
	where we used H\"older's and Young's inequalities. As $\H_0^1(\mathcal{O})\subset\L^{\infty}(\mathcal{O})\subset\L^4(\mathcal{O}),$ the right hand side of the inequality \eqref{214} tends to zero as $n\to\infty$, since $u^n\to u$ in $\H_0^1(\mathcal{O})$. Hence the operator $F(\cdot)$ is demicontinuous, which implies that the operator $F(\cdot)$ is hemicontinuous. 
\end{proof}

\subsection{Stochastic setting} In this subsection, we provide the definition and properties of the noise used in this work and also give the hypothesis satisfied by the noise coefficient.

Let $(\Omega,\mathscr{F},\mathbb{P})$ be a complete probability space equipped with an increasing family of sub-sigma fields $\{\mathscr{F}_t\}_{0\leq t\leq T}$ of $\mathscr{F}$ satisfying  
\begin{enumerate}
	\item [(i)] $\mathscr{F}_0$ contains all elements $F\in\mathscr{F}$ with $\P(F)=0$,
	\item [(ii)] $\mathscr{F}_t=\mathscr{F}_{t+}=\bigcap\limits_{s>t}\mathscr{F}_s,$ for $0\leq t\leq T$.
\end{enumerate}
\begin{definition}
	A stochastic process $\{W(t)\}_{0\leq
		t\leq T}$ is said to be an \emph{$\L^2(\mathcal{O})$-valued $\mathscr{F}_t$-adapted
		Wiener process} with covariance operator $Q$ if
	\begin{enumerate}
		\item [$(i)$] for each non-zero $h\in \H$, $|Q^{\frac{1}{2}}h|^{-1} (W(t), h)$ is a standard one dimensional Wiener process,
		\item [$(ii)$] for any $h\in\L^2(\mathcal{O}), (W(t), h)$ is a martingale adapted to $\mathscr{F}_t$.
	\end{enumerate}
\end{definition}
The stochastic process $\{W(t) : 0\leq t\leq T\}$ is a $\L^2(\mathcal{O})$-valued Wiener process with covariance $Q$ if and only if for arbitrary $t$, the process $W(t)$ can be expressed as $W(t) =\sum\limits_{k=1}^{\infty}\sqrt{\mu_k}e_k(x)\beta_k(t)$, where  $\beta_{k}(t),k\in\mathbb{N}$ are independent one dimensional Brownian motions on $(\Omega,\mathscr{F},\mathbb{P})$ and $\{e_k \}_{k=1}^{\infty}$ are the orthonormal basis functions of $\L^2(\mathcal{O})$ such that $Q e_k=\mu_k e_k$.  If $W(\cdot)$ is an $\L^2(\mathcal{O})$-valued Wiener process with covariance
operator $Q$ such that $\Tr Q=\sum\limits_{k=1}^{\infty} \mu_k< +\infty$, then $W(\cdot)$ is a Gaussian
process on $\L^2(\mathcal{O})$ and $ \E[W(t)] = 0,$ $\textrm{Cov} [W(t)] = tQ,$
$t\geq 0.$ The space $\L^2_Q(\mathcal{O})=Q^{\frac{1}{2}}\L^2(\mathcal{O})$ is a Hilbert space
equipped with the inner product $(\cdot, \cdot)_0$,
$$(u, v)_0
=\sum_{k=1}^{\infty}\frac{1}{\lambda_k}(u,e_k)(v,e_k)=
\left(Q^{-\frac{1}{2}}u, Q^{-\frac{1}{2}}v\right),\ \text{for all } \
u, v\in \L^2_Q(\mathcal{O}),$$ where $Q^{-\frac{1}{2}}$ is the pseudo-inverse of
$Q^{\frac{1}{2}}$.

Let $\mathcal{L}(\L^2(\mathcal{O}))$ denote the space of all
bounded linear operators on $\L^2(\mathcal{O})$ and
$\mathcal{L}_{Q}:=\mathcal{L}_{Q}(\L^2(\mathcal{O}))$ denote the space of all
Hilbert-Schmidt operators from $\L^2_Q(\mathcal{O})=Q^{\frac{1}{2}}\L^2(\mathcal{O})$ to $\L^2(\mathcal{O})$.  Since $Q$ is a trace class operator, the
embedding of $\L^2_Q(\mathcal{O})$ in $\L^2(\mathcal{O})$ is Hilbert-Schmidt and
the space $\mathcal{L}_{Q}$ is a Hilbert space
equipped with the norm
$$
\left\|\Psi\right\|^2_{\mathcal{L}_{Q}}=\Tr\left(\Psi
{Q}\Psi^*\right)=\sum\limits_{k=1}^{\infty}\|
{Q}^{1/2}\Psi^*e_k\|_{\L^2}^2
$$ and inner product $$
\left(\Psi,\Phi\right)_{\mathcal{L}_{Q}}=\Tr\left(\Psi
{Q}\Phi^*\right)=\sum_{k=1}^{\infty}\left({Q}^{1/2}\Phi^*e_k,{Q}^{1/2}\Psi^*e_k\right)
.$$ For more details, the interested readers are referred to see \cite{DaZ}.

Let us assume that the noise coefficient $\sigma(\cdot,\cdot)$  satisfies the following hypothesis.

\begin{hypothesis}\label{hyp}
Let  
	\begin{itemize}
		\item [(H.1)] the noise coefficient $\sigma\in\C([0,T]\times\H_0^1(\mathcal{O});\mathcal{L}_{Q}(\L^2(\mathcal{O})))$,
		\item[(H.2)]  (growth condition)
		there exist a positive
		constant $K$ such that, for all $u\in \L^2(\mathcal{O})$ and $t\in[0,T]$,
		\begin{equation*}
		\|\sigma(t, u)\|^{2}_{\mathcal{L}_{Q}} 
		\leq K\left(1 +\|u\|_{\L^2}^{2}\right),
		\end{equation*}
		
		\item[(H.3)]  (Lipschitz condition)
		there exists a positive constant $L$ such that, for all $u_1, u_2\in \H$ and $t \in [0, T]$,  
		\begin{align*}\|\sigma(t, u_1) - \sigma(t,
		u_2)\|^2_{\mathcal{L}_{Q}}
		\leq L\|u_1 -
		u_2\|_{\L^2}^2.\end{align*}
	\end{itemize}
\end{hypothesis}	

With the above functional setting we rewrite the abstract formulation of the system \eqref{1.1}-\eqref{1.6} as 
\begin{equation}\label{abstract}
\left\{
\begin{aligned}
du(t)&=[- \nu Au(t)-\alpha B(u(t))+\beta c(u(t))]dt+\sigma(t,u(t))dW(t), \ t\in(0,T),\\
u(0)&=u_0,
\end{aligned}
\right.
\end{equation}
where $u_0\in\L^{2p}(\Omega;\L^{2}(\mathcal{O}))$, for $p>2$.

\section{Strong solution}\label{sec3}\setcounter{equation}{0} 
In this section, we prove the existence and uniqueness of strong solution to the system \eqref{abstract} by making use of local monotonicity results obtained in Theorem \ref{monotone} and stochastic generalization of localized version of the Minty-Browder technique.  Let us first  give the definition of a unique global strong solution to the system (\ref{abstract}). 
\begin{definition}[Global strong solution]\label{def3.1}
	Let $u_0\in\L^{2p}(\Omega;\L^2(\mathcal{O})),$ $p>2$ be given. An $\H_0^1(\mathcal{O})$-valued $(\mathscr{F}_t)_{t\geq 0}$-adapted progressively measurable stochastic process $u(\cdot)$ is called a \emph{strong solution} to (\ref{abstract}), if the following conditions are satisfied: 
	\begin{enumerate}
		\item [(i)] the process $$u\in\mathrm{L}^{2p}(\Omega;\L^{\infty}(0,T;\L^2(\mathcal{O})))\cap\L^2(\Omega;\mathrm{L}^2(0,T;\H_0^1(\mathcal{O})))\cap\L^4(\Omega;\L^4(0,T;\L^4(\mathcal{O})))$$ is having continuous modification (still denoted by ${u}$) with ${u}\in\C([0,T];\L^2(\mathcal{O}))\cap\L^2(0,T;\H_0^1(\mathcal{O}))$, $\mathbb{P}$-a.s.,
		\item [(ii)] the following equality holds for every $t\in [0, T ]$, as an element of $\H^{-1}(\mathcal{O}),$ $\mathbb{P}$-a.s.
		\begin{align}\label{4.4}
	u(t)&=u_0+\int_0^t[- \nu Au(s)-\alpha B(u(s))+\beta c(u(s))]ds+\int_0^t\sigma(s,u(s))dW(s).
		\end{align}
	\end{enumerate}
\end{definition}
An alternative version of condition (\ref{4.4}) is to require that for any  $v\in\H_0^1(\mathcal{O})$:
\begin{align}\label{4.5}
(	u(t),v)&=(u_0,v)+\int_0^t\langle- \nu Au(s)-\alpha B(u(s))+\beta c(u(s)),v\rangle ds+\int_0^t(\sigma(s,u(s))dW(s),v).
\end{align}	
\begin{definition}
	A strong solution $u(\cdot)$ to (\ref{abstract}) is called a
	\emph{pathwise unique strong solution} if
	$\widetilde{u}(\cdot)$ is an another strong
	solution, then  $$\mathbb{P}\Big\{\omega\in\Omega:u(t)=\widetilde{u}(t),\text{ for all } t\in[0,T]\Big\}=1.$$ 
\end{definition}

\subsection{Energy estimates}
Let us first show the energy estimates satisfied by the system \eqref{abstract}. 
Let the functions $w_k=w_k(x),$ $k=1,2,\ldots,$  be smooth, the set $\{w_k(x)\}_{k=1}^{\infty}$ be an orthogonal basis of  $\H_0^1(\mathcal{O})$ and orthonormal basis of $\H=\L^2(\mathcal{O})$. One can take $\{w_k(x)\}_{k=1}^{\infty}$ as the complete set of normlized eigenfunctions of the operator $-\partial_{xx}$ in $\H_0^1(\mathcal{O})$.  Let $\H_n$ be the
$n$-dimensional subspace of $\H$. Let $P_n$ denote
the orthogonal projection of $\H^{-1}(\mathcal{O})$ onto $\H_n$, that is, $P_nx=\sum\limits_{k=1}^n\langle x,w_k\rangle w_k$. Since every element $x\in\H$ induces a functional $x^*\in\H$  by the formula $\langle x^*,y\rangle=(x,y)$, $y\in\H_0^1(\mathcal{O})$, then $P_n\big|_{\H}$, the orthogonal projection of $\H$ onto $\H_n$  is given by $P_nx=\sum\limits_{k=1}^n(x,w_k)w_k$. Hence in particular, $P_n$ is the orthogonal projection from $\H$ onto $\text{span}\{w_1,\ldots,w_n\}$.  We define $B_n(u_n)=P_nB(u_n)$, $c_n(u_n)=P_nc(u_n)$, 
$W_n(\cdot)=P_nW(\cdot)$,  and
$\sigma_n(\cdot,u_n)=P_n\sigma(\cdot,u_n)$.
Let us now consider the following system of ODEs:
\begin{equation}\label{4.7}
\left\{
\begin{aligned}
d(u_n(t),v)&=(-\nu Au_n(t)-\alpha B_n(u_n(t))+\beta c_n(u_n(t)),v) d t+(\sigma_n(t,u_n(t))d W_n(t),v),\\
u_n(0)&=u_0^n,
\end{aligned}
\right.
\end{equation}
with $u_0^n=P_nu_0,$ for all $v\in\H_n$. Since $\B_n(\cdot)$ and $c_n(\cdot)$ are locally Lipschitz (see \eqref{2.1} and \eqref{2p7}), and  $\sigma_n(\cdot,\cdot)$  is globally Lipschitz, the system (\ref{4.7}) has a unique $\H_n$-valued local solution $u_n(\cdot)$ and $u_n\in\mathrm{L}^2(\Omega;\mathrm{L}^{\infty}(0,T^*;\H_n))$ with continuous sample paths. 
Let us now derive a-priori energy estimates satisfied by the system \eqref{4.7}. 
\begin{proposition}[Energy estimate]\label{prop1}
	Let $u_n(\cdot)$ be the unique solution of the system of stochastic
	ODE's (\ref{4.7}) with $u_0\in\L^{2p}(\Omega;\L^2(\mathcal{O}))$, $p>2$. Then, we have 
	\begin{align}\label{energy1}
	&\E\left[\sup_{t\in[0,T]}\|u_n(t)\|_{\L^2}^2+4\nu\int_0^{T}\|u_n(t)\|_{\H_0^1}^2d t+2\beta\int_0^{T}\|u_n(t)\|_{\L^4}^4dt\right]\nonumber\\&\quad \leq
(2\|u_0\|_{\L^2}^2+14KT)e^{4(\beta(1+\gamma^2)+7K)T}.\\
	&\E\bigg[\sup_{t\in[0,T]}\|u_n(t)\|_{\L^2}^{2p}+4p\nu\int_0^{T}\|u_n(t)\|_{\L^2}^{2(p-1)}\|u_n(t)\|_{\H_0^1}^2d t+2p\beta\int_0^{T}\|u_n(t)\|_{\L^2}^{2(p-1)}\|u_n(t)\|_{\L^4}^4dt\bigg]\nonumber\\&\quad \leq \left[2\|u_0\|_{\L^2}^{2p} +C(p,K,T)2^{p}T\right]e^{\left\{[2p\beta(1+\gamma^2)+C(p,K,T)2^{p}T]T\right\}},\label{energy2}
	\end{align}
	where $C(p,K,T)=(4(p-1))^{p-1}(14p-1)^pK^p$. 
\end{proposition}
\begin{proof}
Let us define a sequence of stopping times $\tau_N$ by
	\begin{align}\label{stopm}
	\tau_N^n:=\inf_{t\geq 0}\left\{t:\|u_n(t)\|_{\L^2}^2+\int_0^t\|u_n(s)\|_{\H_0^1}^2d s+\int_0^t\|u_n(s)\|_{\L^4}^4ds\geq N\right\},
	\end{align}
	for $N\in\mathbb{N}$. Applying  the finite dimensional It\^{o} formula to the process
	$\|u_n(\cdot)\|_{\L^2}^2$, we obtain (see Theorem 32, \cite{PEP})
	\begin{align}\label{3.6}
	\|u_n(\t)\|_{\L^2}^2&=
	\|u_n(0)\|_{\L^2}^2+2\int_0^{\t}\langle-\nu Au_n(s)-\alpha B_n(u_n(s))+\beta c_n(u_n(s)),u_n(s)\rangle d s
	\nonumber\\&\quad+\int_0^{\t}\|\sigma_n(s,u_n(s))\|^2_{\mathcal{L}_{Q}}d
	s 
	+2\int_0^{\t}\left(\sigma_n(s,u_n(s))dW_n(s),u_n(s)\right).
	\end{align}
	Note that $\langle\B_n(u_n),u_n\rangle=\langle\B(u_n),u_n\rangle=0$, using \eqref{6}. Using \eqref{7}, we estimate $(c_n(u_n),u_n)$ as 
	\begin{align}\label{3p7}
	(c(u_n),u_n)&\leq\frac{(1+\gamma^2)}{2}\|u_n\|_{\L^2}^2-\frac{1}{2}\|u_n\|_{\L^4}^4.
	\end{align}
	Let us use \eqref{3p7} in \eqref{3.6} and then take expectation to get 
	\begin{align}\label{3p8}
	&\E\left[\|u_n(\t)\|_{\L^2}^2+2\nu\int_0^{\t}\|u_n(s)\|_{\H_0^1}^2d s+\beta\int_0^{\t}\|u_n(s)\|_{\L^4}^4ds\right]\nonumber\\&\leq
	\|u_0\|_{\L^2}^2+\beta(1+\gamma^2)\E\left[\int_0^{\t}\|u_n(s)\|_{\L^2}^2ds\right] +\E\left[\int_0^{\t}\|\sigma_n(s,u_n(s))\|^2_{\mathcal{L}_{Q}}d
	s \right],
	\end{align}
where we used $\|u_n(0)\|_{\L^2}\leq \|u_0\|_{\L^2}$ and the fact that final term in the right hand side of (\ref{3.6}) is a local martingale. 	Using the Hypothesis \ref{hyp}  (H.2) in \eqref{3p8}, we obtain 
\begin{align}\label{3p9}
	&\E\left[\|u_n(\t)\|_{\L^2}^2+2\nu\int_0^{\t}\|u_n(s)\|_{\H_0^1}^2d s+\beta\int_0^{\t}\|u_n(s)\|_{\L^4}^4ds\right]\nonumber\\&\leq
\|u_0\|_{\L^2}^2+KT+\left[\beta(1+\gamma^2)+K\right]\E\left[\int_0^{t}\chi_{[0,\t)}(s)\|u_n(s)\|_{\L^2}^2ds\right],
\end{align}
An application of Gronwall's inequality in (\ref{3p9}) yields
	\begin{align}\label{3p10}
	&\E\left[\|u_n(\t)\|_{\L^2}^2\right]\leq
\left(\|u_0\|_{\L^2}^2+KT\right)e^{[\beta(1+\gamma^2)+K]T},
	\end{align}
	for all $t\in[0,T]$. It can be shown that 
	\begin{align}\label{3p11}
	\lim_{N\to\infty}\P\Big\{\omega\in\Omega:\tau_N(\omega)<t\Big\}=0, \ \textrm{
		for all }\ t\in [0,T],
	\end{align}
	and $\t\to t$ as $N\to\infty$. 
	On taking limit $N\to\infty$ in
	(\ref{3p10}) and using the \emph{monotone convergence theorem}, we get 
	\begin{align}\label{3p12}
\sup_{t\in[0,T]}	\E\left[\|u_n(t)\|_{\L^2}^2\right]\leq
\left(\|u_0\|_{\L^2}^2+KT\right)e^{[\beta(1+\gamma^2)+K]T}.
	\end{align}
Substituting (\ref{3p12}) in (\ref{3p9}), we finally arrive at
	\begin{align}\label{4.16a}
	&\E\left[\|u_n(t)\|_{\L^2}^2+2\nu\int_0^{t}\|u_n(s)\|_{\H_0^1}^2d s+\beta\int_0^{t}\|u_n(s)\|_{\L^4}^4ds\right]\leq
\left(\|u_0\|_{\L^2}^2+KT\right)e^{2[\beta(1+\gamma^2)+K]T},
	\end{align}
	for $t\in[0,T]$. Note that the right hand side of the inequality \eqref{4.16a} is independent of $n$. 

Let us  take supremum from $0$
	to $\T$ before taking expectation in (\ref{3p8})  to obtain
	\begin{align}\label{4.17}
	&\E\left[\sup_{t\in[0,\T]}\|u_n(t)\|_{\L^2}^2+2\nu\int_0^{\T}\|u_n(t)\|_{\H_0^1}^2d t+\beta\int_0^{\T}\|u_(t)\|_{\L^4}^4dt\right]\nonumber\\&\leq
	\|u_0\|_{\L^2}^2+\beta(1+\gamma^2)\E\left[\int_0^{\T}\|u_n(t)\|_{\L^2}^2dt\right]
	+\E\left[\int_0^{\T}\|\sigma_n(t,u_n(t))\|^2_{\mathcal{L}_{Q}}d t\right]  \nonumber\\&\quad +2\E\left[\sup_{t\in[0,\T]}\left|\int_0^{t}\left(\sigma_n(s,u_n(s))dW_n(s),u_n(s)\right)\right|\right].
	\end{align}
	Let us take the final term from the right hand side of the inequality (\ref{4.17}) and use Burkholder-Davis-Gundy (see Theorem 1, \cite{BD} and Theorem 1.1, \cite{DLB}), H\"{o}lder's and Young's inequalities to get
	\begin{align}\label{4.18}
&2\E\left[\sup_{t\in[0,\T]}\left|\int_0^{t}\left(\sigma_n(s,u_n(s))dW_n(s),u_n(s)\right)\right|\right]\nonumber\\&\leq 2\sqrt{3}\E\left[\int_0^{\T}\|\sigma_n(t,u_n(t))\|_{\mathcal{L}_{Q}}^2\|u_n(t)\|_{\L^2}^2d t\right]^{1/2}\nonumber\\&\leq 2 \sqrt{3}\E\left[\sup_{t\in[0,\T]}\|u_n(t)\|_{\L^2}\left(\int_0^{\T}\|\sigma_n(t,u_n(t))\|_{\mathcal{L}_{Q}}^2d t\right)^{1/2}\right]\nonumber\\&\leq \frac{1}{2} \E\Bigg[\sup_{t\in[0,\T]}\|u_n(t)\|_{\L^2}^2\Bigg]+6\E\Bigg[\int_0^{\T}\|\sigma_n(t,u_n(t))\|^2_{\mathcal{L}_{Q}}d
	t\Bigg].
	\end{align}
	Substituting (\ref{4.18})  in (\ref{4.17}), we find
	\begin{align}\label{4.20}
	&\E\left[\sup_{t\in[0,\T]}\|u_n(t)\|_{\L^2}^2+4\nu\int_0^{\T}\|u_n(t)\|_{\H_0^1}^2d t+2\beta\int_0^{\T}\|u_(t)\|_{\L^4}^4dt\right]\nonumber\\&\leq 2\|u_0\|_{\L^2}^2+14KT+(2\beta(1+\gamma^2)+14K)\E\left[\int_0^T\chi_{[0,\T)}(t)\|u_n(t)\|_{\L^2}^2d t\right],
	\end{align}
	where  we used the Hypothesis \ref{hyp} (H.2). Applying Gronwall's inequality in (\ref{4.20}), we obtain 
	\begin{align}\label{4.21}
	\E\left[\sup_{t\in[0,\T]}\|u_n(t)\|_{\L^2}^2\right]\leq (2\|u_0\|_{\L^2}^2+14KT)e^{(2\beta(1+\gamma^2)+14K)T}.
	\end{align}
	Passing $N\to\infty$, using the monotone convergence theorem and then substituting (\ref{4.21}) in (\ref{4.20}), we finally obtain the energy estimate in (\ref{energy1}). 
	
In order to prove the  estimate \eqref{energy2}, we apply the finite dimensional It\^o's formula to the process $\|u_n(\cdot)\|_{\L^2}^{2p}$ to obtain 
	\begin{align}\label{4.23a}
	&\|u_n(t)\|_{\L^2}^{2p}+2p\nu\int_0^{t}\|u_n(s)\|_{\L^2}^{2(p-1)}\|u_n(s)\|_{\H_0^1}^2d s\nonumber\\&=
	\|u_n(0)\|_{\L^2}^{2p}+2p\int_0^{t}\|u_n(s)\|_{\L^2}^{2(p-1)}\langle(-\alpha B_n(u_n(s))+\beta c_n(u_n(s)),u_n(s)\rangle d s
	\nonumber\\&\quad
	+2p\int_0^{t}\|u_n(s)\|_{\L^2}^{2(p-1)}\left(\sigma_n(s,u_n(s))dW_n(s),u_n(s)\right)
 \nonumber\\&\quad +p(2p-1)\int_0^t\|u_n(s)\|^{2(p-1)}_{\L^2}\textrm{Tr}(\sigma(s,u_n(s))
	Q\sigma(s,u_n(s)))d s.
	\end{align}
	We take supremum over time from $0$ to $\T$, take expectation in (\ref{4.23a})  to obtain 
	\begin{align}\label{4.24a}
	&\E\bigg[\sup_{t\in[0,\T]}\|u_n(t)\|_{\L^2}^{2p}+2p\nu\int_0^{\T}\|u_n(t)\|_{\L^2}^{2(p-1)}\|u_n(t)\|_{\H_0^1}^2d t+2p\beta\gamma \int_0^{\T}\|u_n(t)\|_{\L^2}^{2p}dt\nonumber\\&\qquad+2p\beta\int_0^{\T}\|u_n(t)\|_{\L^2}^{2(p-1)}\|u_n(t)\|_{\L^4}^4dt\bigg]\nonumber\\&\leq 
\E\left[\|u_0\|_{\L^2}^{2p}\right]+2p\beta(1+\gamma)\E\left[\int_0^{\T}\|u_n(t)\|_{\L^2}^{2(p-1)}(u_n^2(t),u_n(t))d t\right]
	\nonumber\\&\quad
	+2p\E\left[\sup_{t\in[0,\T]}\left|\int_0^{t}\|u_n(s)\|_{\L^2}^{2(p-1)}\left(\sigma_n(s,u_n(s))dW_n(s),u_n(s)\right)\right|\right]
	\nonumber\\&\quad+p(2p-1)\E\left[\int_0^{\T}\|u_n(t)\|_{\L^2}^{2(p-1)}\|\sigma_n(t,u_n(t))\|^2_{\mathcal{L}_{Q}}d
		t\right]\nonumber\\&=\sum_{i=1}^3J_i,
	\end{align}
	where $J_i$, for $i=1,2,3$ are the final three terms appearing the right hand side of the inequality \eqref{4.24a}.
	Let us use H\"older's  and Young's inequalities to estimate $J_1$ as 
	\begin{align}\label{3.54z}
	J_1&\leq 2p\beta(1+\gamma)\E\left[\int_0^{\T}\|u_n(t)\|_{\L^2}^{2(p-1)}\|u_n(t)\|_{\L^4}^2\|u_n(t)\|_{\L^2}d t\right]\nonumber\\&\leq 2p\beta(1+\gamma)\E\left[\left(\int_0^{\T}\|u_n(t)\|_{\L^2}^{2(p-1)}\|u_n(t)\|_{\L^4}^4dt\right)^{1/2}\left(\int_0^{\T}\|u_n(t)\|_{\L^2}^{2p}dt\right)^{1/2}\right]\nonumber\\&\leq p\beta\E\left[\int_0^{\T}\|u_n(t)\|_{\L^2}^{2(p-1)}\|u_n(t)\|_{\L^4}^4dt\right]+p\beta(1+\gamma)^2\E\left[\int_0^{\T}\|u_n(t)\|_{\L^2}^{2p}dt\right].
	\end{align}
	Using Burkholder-Davis-Gundy, H\"older's and Young's inequalities, we estimate $J_2$ as 
	\begin{align}\label{4.31z}
	J_2&\leq 2p\sqrt{3}\E\left[\int_0^{\T}\|u_n(t)\|_{\L^2}^{4p-2}\|\sigma_n(t,u_n(t))\|_{\mathcal{L}_{Q}}^2d t\right]^{1/2}\nonumber\\&\leq  2p\sqrt{3}\E\left[\sup_{t\in[0,\T]}\|u_n(t)\|_{\L^2}^{p}\left(\int_0^{\T}\|u_n(t)\|_{\L^2}^{2(p-1)}\|\sigma_n(t,u_n(t))\|_{\mathcal{L}_{Q}}^2d t\right)^{1/2}\right]\nonumber\\&\leq \frac{1}{4}\E\left[\sup_{t\in[0,\T]}\|u_n(t)\|_{\L^2}^{2p}\right]+12p^2\E\left[\int_0^{\T}\|u_n(t)\|_{\L^2}^{2(p-1)}\|\sigma_n(t,u_n(t))\|_{\mathcal{L}_{Q}}^2d t\right]\nonumber\\&\leq\frac{1}{4}\E\left[\sup_{t\in[0,\T]}\|u_n(t)\|_{\L^2}^{2p}\right]+J_4.
	\end{align}
	Let us use H\"older's and Young's inequalities to estimate $J_3+J_4$ as 
	\begin{align}\label{4.32z}
	J_3+J_4&\leq p(14p-1)\E\left[\sup_{t\in[0,T]}\|u_n(t)\|_{\L^2}^{2(p-1)}\int_0^{\T}\|\sigma_n(t,u_n(t))\|_{\mathcal{L}_{Q}}^2d t\right]\nonumber\\&\leq\frac{1}{4}\E\left[\sup_{t\in[0,\T]}\|u_n(t)\|_{\L^2}^{2p}\right]\nonumber\\&\qquad+\frac{1}{p}\left(\frac{4(p-1)}{p}\right)^{p-1}(p(14p-1))^{p}\E\left[\int_0^{\T}\|\sigma_n(t,u_n(t))\|_{\mathcal{L}_{Q}}^2d t\right]^p\nonumber\\&\leq\frac{1}{4}\E\left[\sup_{t\in[0,\T]}\|u_n(t)\|_{\L^2}^{2p}\right]+(4(p-1))^{p-1}(14p-1)^pK^p\E\left[\int_0^{\T}(1+\|u_n(s)\|_{\L^2}^2)ds\right]^p\nonumber\\&\leq \frac{1}{4}\E\left[\sup_{t\in[0,\T]}\|u_n(t)\|_{\L^2}^{2p}\right]+C(p,K,T)\E\left[\int_0^{\T}(1+\|u_n(s)\|_{\L^2}^2)^pds\right]\nonumber\\&\leq \frac{1}{4}\E\left[\sup_{t\in[0,\T]}\|u_n(t)\|_{\L^2}^{2p}\right]+C(p,K,T)2^{p-1}T\nonumber\\&\quad+C(p,K,T)2^{p-1}\E\left[\int_0^{\T}\|u_n(s)\|_{\L^2}^{2p}ds\right],
	\end{align}
	where $C(p,K,T)=(4(p-1))^{p-1}(14p-1)^pK^p$. 
	Combining (\ref{3.54z})-(\ref{4.32z}) and using it in (\ref{4.24a}), we arrive at
	\begin{align}\label{4.35z}
	&\E\bigg[\sup_{t\in[0,\T]}\|u_n(t)\|_{\L^2}^{2p}+4p\nu\int_0^{\T}\|u_n(t)\|_{\L^2}^{2(p-1)}\|u_n(t)\|_{\H_0^1}^2d t\nonumber\\&\qquad+2p\beta\int_0^{\T}\|u_n(t)\|_{\L^2}^{2(p-1)}\|u_n(t)\|_{\L^4}^4dt\bigg]\nonumber\\&\leq  2\|u_0\|_{\L^2}^{2p} +C(p,K,T)2^{p}T +[2p\beta(1+\gamma^2)+C(p,K,T)2^{p}T]\E\left[\int_0^{\T}\|u_n(t)\|_{\L^2}^{2p}dt\right].
	\end{align}
 Let us apply  Gronwall's inequality in (\ref{4.35z})  to get 
	\begin{align}\label{4.36z}
	&\E\left[\sup_{t\in[0,\T]}\|u_n(t)\|_{\L^2}^{2p}\right]\leq \left[2\|u_0\|_{\L^2}^{2p} +C(p,K,T)2^{p}T\right]e^{\left\{[2p\beta(1+\gamma^2)+C(p,K,T)2^{p}T]T\right\}}.
	\end{align}
	Passing $N\to\infty$, using the monotone convergence theorem in (\ref{4.36z}) and then applying it in (\ref{4.35z}), we arrive at (\ref{energy2}).
\end{proof}
\begin{remark}
From \eqref{3.6}, using Hypothesis \ref{hyp} (H.2) and It\^o isometry, we have 
\begin{align}\label{3.26}
&\E\left[\left(\int_0^{T}\|u_n(s)\|_{\H_0^1}^2ds\right)^2\right]\nonumber\\&\leq\E\bigg[ \bigg(\|u_0\|_{\L^2}^2+\beta(1+\gamma^2)\int_0^{T}\|u_n(s)\|_{\L^2}^2ds+\int_0^{T}\|\sigma_n(t,u_n(t))\|^2_{\mathcal{L}_{Q}}dt 
\nonumber\\&\qquad+2\int_0^{T}\left(\sigma_n(t,u_n(t))dW_n(t),u_n(t)\right)\bigg)^2\bigg]\nonumber\\&\leq 4\E\left[\|u_0\|_{\L^2}^4\right]+8KT+4[\beta(1+\gamma^2)T+2KT]\mathbb{E}\left[\sup_{t\in[0,T]}\|u_n(s)\|_{\L^2}^4\right]\nonumber\\&\quad +16\E\left[\int_0^T\|\sigma_n(s,u_n(s))\|_{\mathcal{L}(Q)}^2\|u_n(s)\|_{\L^2}^2ds\right]\nonumber\\&\leq 4\E\left[\|u_0\|_{\L^2}^4\right]+8KT+16KT\mathbb{E}\left[\sup_{t\in[0,T]}\|u_n(s)\|_{\L^2}^2\right]\nonumber\\&\quad+4[\beta(1+\gamma^2)T+18KT]\mathbb{E}\left[\sup_{t\in[0,T]}\|u_n(s)\|_{\L^2}^4\right]<+\infty,
\end{align}
by using \eqref{energy2}. Similarly, we have 
\begin{align}\label{3p27}
\E\left[\left(\int_0^{T}\|u_n(t)\|_{\L^4}^4dt\right)^2\right]&\leq  4\E\left[\|u_0\|_{\L^2}^4\right]+8KT+16KT\mathbb{E}\left[\sup_{t\in[0,T]}\|u_n(t)\|_{\L^2}^2\right]\nonumber\\&\quad+4[\beta(1+\gamma^2)T+18KT]\mathbb{E}\left[\sup_{t\in[0,T]}\|u_n(t)\|_{\L^2}^4\right]<+\infty.
\end{align}
\end{remark}

\subsection{Existence and uniqueness of strong solution} Let us now prove that the system (\ref{abstract}) has a unique global strong solution by exploiting the local monotonicity property (see (\ref{3.11y})) and a stochastic generalization of the Minty-Browder technique. This method is applied  in \cite{ICAM} for establishing  the existence of strong solutions to  stochastic 2D hydrodynamical type systems. Similar existence results for the 2D stochastic Navier-Stokes equations driven by Gaussian noise can be found in \cite{MJSS,SSSP} and stochastic 2D Oldroyd model for viscoelastic models  can be found in \cite{MTM3}. 
\begin{theorem}[Existence and uniqueness of strong solution to the system (\ref{abstract})]\label{exis}
	Let $u_0\in \L^{2p}(\Omega;\L^2(\mathcal{O})),$ for $p>2$ be given.  Then there exists a \emph{unique strong solution}
	$u(\cdot)$ to the problem (\ref{abstract}) such that $$u\in \mathrm{L}^{2p}(\Omega;\mathrm{L}^{\infty}(0,T;\L^2(\mathcal{O})))\cap\L^2(\Omega;\mathrm{L}^2(0,T;\H_0^1(\mathcal{O})))\cap\L^4(\Omega;\L^4(0,T;\L^4(\mathcal{O})))$$  and $u(\cdot)$ is having a $\mathbb{P}$-a.s., continuous modification in $\C([0,T];\L^2(\mathcal{O}))\cap\mathrm{L}^2(0,T;\H_0^1(\mathcal{O}))$.
\end{theorem}
\begin{proof}
	The proof of the solvability results of the system (\ref{abstract}) is divided into the following steps.
	
	\vskip 0.2cm
	\noindent\textbf{Step (1):} \emph{Finite-dimensional (Galerkin) approximation of the system (\ref{abstract}):} 	Let us first consider the following It\^{o} stochastic
	differential equation satisfied by $\{u_n(\cdot)\}$:
	\begin{equation}\label{4.37}
	\left\{
	\begin{aligned}
	du_n(t)&=-F(u_n(t))d
	t+\sigma_n(t,u_n(t))dW_n(t),\\
	u_n(0)&=u_0^n,
	\end{aligned}
	\right.
	\end{equation}
	where
	$F(u_n)=\nu A u_n+\alpha B_n(u_n)-\beta c_n(u_n)$. Applying It\^o's formula to the process $e^{-r(t)}\|u_n(\cdot)\|_{\L^2}^2$, we have the following equality:
	\begin{align}\label{4.38}
	e^{-r(t)}\|u_n(t)\|_{\L^2}^2&=e^{-r(0)}\|u_n(0)\|_{\L^2}^2-\int_0^te^{-r(s)}\langle2F(u_n(s))+r'(s)u_n(s),u_n(s)\rangle d
	s\\&\quad+2\int_0^te^{-r(s)}\left(\sigma_n(s,u_n(s))dW_n(s),u_n(s)\right)
	+\int_0^te^{-r(s)}\|\sigma_n(s,u_n(s))\|_{\mathcal{L}_{Q}}^2d
	s,\nonumber
	\end{align}
	for all $t\in[0,T]$. The quantity $r(t)$ appearing in \eqref{4.38} will be chosen later. Note that the third term from the right hand side of the equality (\ref{4.38}) is a martingale and on taking expectation, we get 
	\begin{align}\label{4.39}
	\E\left[e^{-r(t)}\|u_n(t)\|_{\L^2}^2\right]&=\E\left[e^{-r(0)}\|u_n(0)\|_{\L^2}^2\right]-\E\left[\int_0^te^{-r(s)}\langle 2F(u_n(s))+r'(s)u_n(s),u_n(s)\rangle d
	s\right]\nonumber\\&\quad+\E\left[\int_0^t
	e^{-r(s)}\|\sigma_n(s,u_n(s))\|_{\mathcal{L}_{Q}}^2d s\right],
	\end{align}
	for all $t\in[0,T]$.
	
	\vskip 0.2cm
	
	\noindent\textbf{Step (2):} \emph{Weak convergence of the
		sequences $u_n(\cdot)$, $F(u_n(\cdot))$ and
		$\sigma_n(\cdot,\cdot)$.} 
	We know that 
	$\mathrm{L}^2\left(\Omega;\mathrm{L}^{\infty}(0,T;\L^2(\mathcal{O}))\right)\cong
	\left(\mathrm{L}^{2}\left(\Omega;\mathrm{L}^1(0,T;\L^2(\mathcal{O}))\right)\right)'$, and  $\mathrm{L}^{2}\left(\Omega;\mathrm{L}^1(0,T;\L^2(\mathcal{O}))\right)$ is separable and the spaces  $\mathrm{L}^2(\Omega;\mathrm{L}^2(0,T;\H_0^1(\mathcal{O})))$ and $\mathrm{L}^4(\Omega;\mathrm{L}^4(0,T;\L^4(\mathcal{O})))$ are reflexive.  Using the energy estimates in Proposition \ref{prop1}, and 	Banach-Alaoglu theorem, we can extract a subsequence	$\{u_{n_k}\}$ of $\{u_n\}$, which converges to the following limits	(for simplicity, we denote the index $n_k$ by $n$):
	\begin{equation}\label{4.40}
	\left\{
	\begin{aligned}
	u_n&\xrightarrow{w^{*}} u\textrm{ in
	}\mathrm{L}^2(\Omega;\mathrm{L}^{\infty}(0,T ;\L^2(\mathcal{O}))),\\ u_n&\xrightarrow{w} u\textrm{ in
	}\mathrm{L}^4(\Omega;\mathrm{L}^{4}(0,T ;\L^4(\G))),\\ 
	u_n&\xrightarrow{w} u\textrm{ in
	}\mathrm{L}^2(\Omega;\mathrm{L}^{2}(0,T ;\H_0^1(\mathcal{O}))),\\ 
	u_n(T)&\xrightarrow{w}\eta \in\mathrm{L}^2(\Omega;\L^2(\mathcal{O})),\\
	F(u_n)&\xrightarrow{w} F_0\textrm{ in
	}\mathrm{L}^{1+\eta}(\Omega;\mathrm{L}^2(0,T ;\H^{-1}(\mathcal{O}))),
	\end{aligned}
	\right.
	\end{equation}
	for some $0<\eta\leq \frac{p-2}{p+2}$. 
	The final convergence  in (\ref{4.5}) can be justified as follows: 
	\begin{align}\label{4.41}
	&\E\left[\int_0^T\left\|F(u_n(t))\right\|_{\H^{-1}}^{1+\eta}d t\right]\nonumber\\&\leq C(\eta)\left\{ \nu\E\left[\int_0^T\|Au_n(t)\|_{\H^{-1}}^{1+\eta}d t\right]+\alpha\E\left[\int_0^T\|B_n(u_n(t))\|_{\H^{-1}}^{1+\eta}d t\right]+\beta\E\left[\int_0^T\|c(u_n(t))\|_{\H^{-1}}^{1+\eta}d t\right]\right\}\nonumber\\&\leq C(\eta)\bigg\{ \nu\E\left[\int_0^T\|u_n(t)\|_{\H_0^1}^{1+\eta}dt\right]+\left[\alpha+\beta(1+\gamma)\right]\E\left[\int_0^T\|u_n(t)\|_{\L^4}^{2(1+\eta)}dt\right]\nonumber\\&\qquad+\beta\gamma\E\left[\int_0^T\|u_n(t)\|_{\L^2}^{1+\eta}dt\right]+\beta\E\left[\int_0^T\|u_n(t)\|_{\L^6}^{3(1+\eta)}dt\right]\bigg\}\nonumber\\&\leq C(\eta)\bigg\{ T^{\frac{1-\eta}{2}}\left\{\E\left[\int_0^T\|u_n(t)\|_{\H_0^1}^2d t\right]\right\}^{\frac{1+\eta}{2}}+[\alpha+\beta(1+\gamma)]T^{\frac{1-\eta}{2}}\left\{\E\left[\int_0^T\|u_n(t)\|_{\L^4}^4dt\right]\right\}^{\frac{1+\eta}{2}}\nonumber\\&\quad+\beta\gamma T^{\frac{1-\eta}{2}}\E\left[\int_0^T\|u_n(t)\|_{\L^2}^2dt\right]^{\frac{1+\eta}{2}}+\beta\E\left[\int_0^T\|u_n(t)\|_{\L^6}^{3(1+\eta)}dt\right]\bigg\},
	\end{align}
	where we used \eqref{2p5} and  H\"older's  inequality. The final term from \eqref{4.41} can be 	controlled by Gagliardo-Nirenberg interpolation and H\"older's inequalities as
	\begin{align}
&\E\left[\int_0^T\|u_n(t)\|_{\L^6}^{3(1+\eta)}dt\right]\nonumber\\&\leq C\E\left[\int_0^T\|\partial_xu_n(t)\|_{\L^2}^{(1+\eta)}\|u_n(t)\|_{\L^2}^{2(1+\eta)}dt\right]\nonumber\\&\leq C\mathbb{E}\left[\left(\int_0^T\|u_n(t)\|_{\H_0^1}^2dt\right)^{\frac{1+\eta}{2}}\left(\int_0^T\|u_n(t)\|_{\L^2}^{\frac{4(1+\eta)}{1-\eta}}dt\right)^{\frac{1-\eta}{2}}\right]\nonumber\\&\leq CT^{\frac{1-\eta}{2}}\left\{\mathbb{E}\left[\int_0^T\|u_n(t)\|_{\H_0^1}^2dt\right]\right\}^{\frac{1+\eta}{2}}\left\{\mathbb{E}\left[\sup_{t\in[0,T]}\|u_n(t)\|_{\L^2}^{\frac{4(1+\eta)}{1-\eta}}\right]\right\}^{\frac{1-\eta}{2}}<+\infty,
	\end{align}
using Proposition \eqref{prop1} (see \eqref{energy1} and \eqref{energy2}). Using the Hypothesis \ref{hyp} (H.1) and energy
	estimates given in Proposition \ref{prop1},
	we have
	\begin{align}\label{4.42}
	\E\left[\int_0^{T
	}\|\sigma_n(t,u_n(t))\|_{\mathcal{L}_{Q}}^2\d
	t\right]&\leq K\E\left[\int_0^T\left(1+\|u_n(t)\|_{\L^2}^2\right)d t\right]\nonumber\\&\leq
	KT(2\|u_0\|_{\L^2}^2+14KT)e^{4(\beta(1+\gamma^2)+7K)T}<+\infty.
	\end{align}
	Thus, we can extract a subsequence
	$\{\sigma_{n_k}(\cdot,u_{n_k})\}$ which converge to the
	following limit (denoting the index $n_k$ by $n$):
	\begin{equation}\label{4.43z}
	\sigma_n(\cdot,u_n)P_n\xrightarrow{w}
	\Phi(\cdot)\textrm{ in
	}\mathrm{L}^2(\Omega;\mathrm{L}^2(0,T ;\mathcal{L}_{Q}(\L^2(\mathcal{O})))).
	\end{equation}
	As discussed in Theorem 7.5 \cite{chow}, one can prove that $u(\cdot)$ satisfies the
	It\^{o} stochastic differential:
	\begin{equation}\label{4.44}
	\left\{
	\begin{aligned}
	du(t)&=-F_0(t)d
	t+\Phi(t)dW(t),\\
	u(0)&=u_0.
	\end{aligned}
	\right.
	\end{equation}
	A calculation similar to (\ref{4.39}) yields
	\begin{align}\label{4.45}
	\E\left[e^{-r(t)}\|u(t)\|_{\L^2}^2\right]&=\E\left[e^{-r(0)}\|u_0\|_{\L^2}^2\right]-\E\left[\int_0^te^{-r(s)}\langle 2F_0(s)+r'(s)u(s),u(s)\rangle d
	s\right]\nonumber\\&\quad+\E\left[\int_0^t
	e^{-r(s)}\|\Phi(s)\|_{\mathcal{L}_{Q}}^2d s\right],
	\end{align}
	for all $t\in[0,T]$. Also, it should be noted that the initial value
	$u_n(0)$ converges to $u_0$ strongly in $\L^2(\Omega;\L^2(\mathcal{O}))$, that is,
	\begin{align}\label{4.46}
	\lim_{n\to\infty}\E\left[\|u_n(0)-u_0\|_{\L^2}^2\right]=0.
	\end{align}

	\noindent\textbf{Step (3):} \emph{Minty-Browder
		technique and global strong solution.} Let us now prove that $F(u(\cdot))=F_0(\cdot)$ and
	$\sigma(\cdot,u(\cdot))=\Phi(\cdot)$. For
	$v\in\mathrm{L}^2(\Omega;\mathrm{L}^{2}(0,T;\H_m))$ with $m<n$, let us
	define
	\begin{align}\label{4.47}
	r(t)=\frac{\alpha^2}{\nu}\int_0^t\|v(s)\|_{\L^{\infty}}^2ds+[2\beta(1+\gamma+\gamma^2)+L]t,
	\end{align}
	so that 
	\begin{align*}
	r'(t)=\frac{\alpha^2}{\nu}\|v(t)\|_{\L^{\infty}}^2+[2\beta(1+\gamma+\gamma^2)+L], \ \text{ a.e.}
	\end{align*}
	From the local
	monotonicity result (see (\ref{3.11y})),
	we have
	\begin{align}\label{4.48}
	&\E\bigg[\int_0^{T}e^{-r(t)}\Big(2\langle F(v(t))- F(u_n(t)),v(t)-u_n(t)\rangle
	+r'(t)\left(v(t)-u_n(t),v(t)-u_n(t)\right)\Big)d
	t\bigg]\nonumber\\&\geq \E\left[\int_0^{T}e^{-r(t)}\|\sigma_n(t,
	v(t)) - \sigma_n(t,u_n(t))\|^2_{\mathcal{L}_{Q}}d
	t\right].
	\end{align}
	In (\ref{4.48}), rearranging the terms and using energy	equality (\ref{4.39}), we get
	\begin{align}\label{4.49}
	&\E\left[\int_0^{T}e^{-r(t)}\langle 2F(v(t))+r'(t)v(t),v(t)-u_n(t)\rangle d
	t\right]\nonumber\\&\quad-\E\left[\int_0^{T}e^{-r(t)}\|\sigma_n(t,
	v(t))\|^2_{\mathcal{L}_{Q}}
	d
	t\right]+2\E\left[\int_0^{T}e^{-r(t)}\left(\sigma_n(t,
	v(t)),
	\sigma_n(t,u_n(t))\right)_{\mathcal{L}_{Q}}d
	t\right]\nonumber\\&\geq
	\E\left[\int_0^{T}e^{-r(t)}\langle 2F(u_n(t))+r'(t)u_n(t),v(t)\rangle d
	t\right]\nonumber\\&\quad-\E\left[\int_0^{T}e^{-r(t)}\langle 2F(u_n(t))+r'(t)u_n(t),u_n(t)\rangle d
	t\right]+\E\left[\int_0^{T}e^{-r(t)}\|
	\sigma_n(t,u_n(t))\|^2_{\mathcal{L}_{Q}}d
	t\right]\nonumber\\&=\E\left[\int_0^{T}e^{-r(t)}\langle 2F(u_n(t))+r'(t)u_n(t),v(t)\rangle d
	t\right]
	+\E\left[e^{-r(T)}\|u_n(T)\|_{\L^2}^2-\|u_n(0)\|_{\L^2}^2\right].
	\end{align}
	We use the weak convergence in	(\ref{4.43z}) 	and Lebesgue dominated convergence theorem to  deduce that
	\begin{align}\label{4.50}
	&\E\Bigg[\int_0^{T}e^{-r(t)}\left(2\left(\sigma_n(t, v(t)),
	\sigma_n(t,u_n(t))\right)_{\mathcal{L}_{Q}}-\|\sigma_n(t,
	v(t))\|^2_{\mathcal{L}_{Q}}\right)d
	t\Bigg]\nonumber\\& \to
	\E\left[\int_0^{T}e^{-r(t)}\left(2\left(\sigma(t, v),
	\Phi(t)\right)_{\mathcal{L}_{Q}}-\|\sigma(t,
	v(t))\|^2_{\mathcal{L}_{Q}}\right)d t\right],
	\end{align}
	as $n\to\infty$. On taking liminf on both sides of (\ref{4.49}),
	and using (\ref{4.50}), we obtain
	\begin{align}\label{4.51}
	&\E\left[\int_0^{T}e^{-r(t)}\langle 2F(v(t))+r'(t)v(t),v(t)-u(t)\rangle d
	t\right]\nonumber\\&\quad-\E\left[\int_0^{T}e^{-r(t)}\|\sigma(t,
	v(t))\|^2_{\mathcal{L}_{Q}} d
	t\right]+2\E\left[\int_0^{T}e^{-r(t)}\left(\sigma(t,
	v(t)), \Phi(t)\right)_{\mathcal{L}_{Q}}d
	t\right]\\&\geq\E\left[\int_0^{T}e^{-r(t)}\langle 2F_0(t)+r'(t)u(t),v(t)\rangle d
	t\right]
	+\liminf_{n\to\infty}\E\left[e^{-r(T)}\|u^n(T)\|_{\L^2}^2-\|u^n(0)\|_{\L^2}^2\right].\nonumber
	\end{align}
	Using the lower semicontinuity property of the $\L^2$-norm and
	(\ref{4.46}), the second term on the right hand side of the
	inequality (\ref{4.51}) satisfies the following inequality:
	\begin{align}\label{4.52}
	&\liminf_{n\to\infty}\E\left[e^{-r(T)}\|u^n(T)\|_{\L^2}^2-\|u^n(0)\|_{\L^2}^2\right]\geq
	\E\left[e^{-r(T)}\|u(T)\|^2_{\L^2}-\|u_0\|^2_{\L^2}\right].
	\end{align}
	Hence by using the energy equality (\ref{4.45}) and (\ref{4.52}) in
	(\ref{4.51}), we find
	\begin{align}\label{4.53}
	&\E\left[\int_0^{T}e^{-r(t)}\langle 2F(v(t))+r'(t)v(t),v(t)-u(t)\rangle d
	t\right]\nonumber\\&\geq\E\left[\int_0^{T}e^{-r(t)}\|\sigma(t,
	v(t))\|^2_{\mathcal{L}_{Q}} d
	t\right]-2\E\left[\int_0^{T}e^{-r(t)}\left(\sigma(t,
	v(t)), \Phi(t)\right)_{\mathcal{L}_{Q}}d
	t\right]\nonumber\\&\quad+\E\left[\int_0^{T}e^{-r(t)}\|\Phi(t)\|^2_{\mathcal{L}_{Q}}d t\right]
	+\E\left[\int_0^{T}e^{-r(t)}\langle 2F_0(t)+r'(t)u(t),v(t)-u(t)\rangle d
	t\right].
	\end{align}
	Thus, by rearranging the terms in (\ref{4.53}), we finally obtain
	\begin{align}\label{4.54}
	&\E\left[\int_0^{T}e^{-r(t)}\langle 2F(v(t))-2F_0(t)+r'(t)(v(t)-u(t)),v(t)-u(t)\rangle d
	t\right]\nonumber\\&\geq
	\E\Bigg[\int_0^{T}e^{-r(t)}\|\sigma(t,
	v(t))-\Phi(t)\|^2_{\mathcal{L}_{Q}}
	d t\Bigg]\geq 0.
	\end{align}
	The estimate (\ref{4.54}) holds true for any	$v\in\mathrm{L}^2(\Omega;\mathrm{L}^{2}(0,T;\H_m))$ and for any
	$m\in\mathbb{N}$, since the estimate (\ref{4.54}) is independent of
	$m$ and $n$. It can be easily seen by a density argument that the
	inequality (\ref{4.54}) remains true for any
	$$v\in\mathrm{L}^{2p}(\Omega;\mathrm{L}^{\infty}(0,T;\L^2(\G)))\cap\L^2(\Omega;\mathrm{L}^2(0,T;\H_0^1(\mathcal{O})))\cap\L^4(\Omega;\mathrm{L}^4(0,T;\L^4(\G)))=:\mathcal{J},$$ for $p>2$.
	Indeed, for any $v\in\mathcal{J}$,
	there exists a strongly convergent subsequence
	$v_m\in\mathcal{J}$ that satisfies the
	inequality (\ref{4.54}).
	Taking $v(\cdot)=u(\cdot)$ in (\ref{4.54}) immediately gives
	$\sigma(\cdot,v(\cdot))=\Phi(\cdot)$. Let us now take
	$v(\cdot)=u(\cdot)+\lambda w(\cdot)$, $\lambda>0$, where
	$w\in\mathrm{L}^{4}(\Omega;\mathrm{L}^{\infty}(0,T;\L^2(\G)))\cap\mathrm{L}^2(\Omega;\mathrm{L}^2(0,T;\H_0^1(\G))),$ and substitute for
	$v$ in (\ref{4.54}) to find
	\begin{align}\label{4.55}
	\E\left[\int_0^{T}e^{-r(t)}\langle 2F(u(t)+\lambda w(t))-2F_0(t)+r'(t)\lambda w(t),\lambda w(t)\rangle d
	t\right]\geq 0.
	\end{align}
	Let us divide the inequality (\ref{4.55}) by $\lambda$,  use the hemicontinuity property of
	$F(\cdot)$, and let $\lambda\to 0$ to obtain
	\begin{align}\label{4.56}
	\E\left[\int_0^{T}e^{-r(t)}\langle F(u(t))-F_0(t),w(t)\rangle d
	t\right]\geq 0.
	\end{align}
	The final term from (\ref{4.55}) tends to $0$ as $\lambda\to0$,
	since
	\begin{align}\label{4.57}
	&\E\left[\int_0^{T}e^{-r(t)}r'(t)\left(w(t),w(t)\right)d
	t\right]\nonumber\\&=\E\left[\int_0^{T}e^{-r(t)}\left\{\frac{\alpha^2}{\nu}\|v(t)\|_{\L^{\infty}}^2+2\beta(1+\gamma+\gamma^2)+L\right\}\|w(t)\|_{\L^2}^2d
	t\right]\nonumber\\&\leq \frac{C\alpha^2}{\nu}\E\left[\sup_{t\in[0,T]}\|w(t)\|_{\L^2}^2\int_0^T\|v(t)\|_{\H_0^1}^2dt\right]+[2\beta(1+\gamma+\gamma^2)+L]\E\left[\int_0^T\|w(t)\|_{\L^2}^2dt\right]\nonumber\\&\leq \frac{C\alpha^2}{\nu}\left\{\E\left[\sup_{t\in[0,T]}\|w(t)\|_{\L^2}^4\right]\right\}^{1/2}\left\{\E\left[\int_0^T\|v(t)\|_{\H_0^1}^2dt\right]^2\right\}^{1/2}\nonumber\\&\quad+[2\beta(1+\gamma+\gamma^2)+L]\E\left[\int_0^T\|w(t)\|_{\L^2}^2dt\right]<+\infty,
	\end{align}
	by using \eqref{3.26} and \eqref{energy2}. Thus from (\ref{4.56}), we have $\F(u(t))=\F_0(t)$ and hence $u(\cdot)$ is a strong solution of the
	system (\ref{abstract}) and
	$u\in\mathcal{J}$. It is clear that $u(\cdot)$ has a modification, whose $\mathscr{F}_t$-adapted paths  are continuous with trajectories in  $\C([0,T];\L^2(\G))$, $\mathbb{P}$-a.s. (see \cite{Me}).
	
	\vskip 0.2cm
	\noindent\textbf{Step (4):} \emph{Pathwise uniqueness.} Let $u_1(\cdot)$ and $u_2(\cdot)$ be two solutions of the system (\ref{abstract}). For $N>0$, let us define 
	\begin{align*}
	\tau_N^1=\inf_{0\leq t\leq T}\Big\{t:\|u_1(t)\|_{\L^2}\geq N\Big\},\ \tau_N^2=\inf_{0\leq t\leq T}\Big\{t:\|u_2(t)\|_{\L^2}\geq N\Big\}\text{ and }\tau_N=\tau_N^1\wedge\tau_N^2.
	\end{align*}
	One can show that $\tau_N\to T$ as $N\to\infty$, $\mathbb{P}$-a.s. Let us take $w(\cdot)=u_1(\cdot)-u_2(\cdot)$ and $\widetilde{\sigma}(\cdot)=\sigma(\cdot,u_1(\cdot))-\sigma(\cdot,u_2(\cdot))$. Then, $w(\cdot)$ satisfies the system 
	\begin{equation}
	\left\{
	\begin{aligned}
	dw(t)&=\left[-\nu Aw(t)-\alpha (B(u_1(t))-B(u_2(t)))+\beta(c(u_1(t))-c(u_2(t)))\right]d t\\&\quad+\widetilde{\sigma}(t)dW(t),\\
	w(0)&=w_0.
	\end{aligned}
	\right.
	\end{equation}
	We apply the infinite dimensional It\^o's formula (see \cite{IG}, Theorem 6.1, \cite{Me}) to the process $e^{-\rho(t)}\|w(t)\|_{\L^2}^2,$ where 
	\begin{align}\label{3pp51}\rho(t)=\frac{C\alpha^2}{\nu}\int_0^t\|u_2(s)\|_{\H_0^1}^2ds,\text{  so that }\ \rho'(t)= \frac{C\alpha^2}{\nu}\|u_2(t)\|_{\H_0^1}^2, \ \text{ a.e.},\end{align}
	to find 
	\begin{align}\label{4.59}
	&e^{-\rho(\s)}\|w(\s)\|_{\L^2}^2+2\nu\int_0^{\s}e^{-\rho(s)}\|w(s)\|_{\H_0^1}^2d s\nonumber\\&=\|w(0)\|_{\L^2}^2 -\int_0^{\t}\rho'(s)e^{-\rho(s)}\|w(s)\|_{\L^2}^2d s-2\alpha\int_0^{\s}e^{-\rho(s)}(u_2(s),w(s)\partial_xw(s))ds\nonumber\\&\quad +2\beta\int_0^{\s}e^{-\rho(s)}(c(u_1(s))-c(u_2(s)),u_1(s)-u_2(s))ds+\int_0^{\s}e^{-\rho(s)}\|\wi\sigma(s)\|_{\mathcal{L}_{Q}}^2d s\nonumber\\&\quad+2\int_0^{\s}e^{-\rho(s)}\left(\widetilde{\sigma}(s)dW(s),w(s)\right),
	\end{align}
where we used \eqref{2.10} and \eqref{2.11}. We estimate the terms $-2\alpha(u_2,w\partial_xw)$ using H\"older's, Gagliardo-Nirenberg interpolation and Young's inequalities as 
\begin{align}\label{3p51}
-2\alpha(u_2,w\partial_xw)&\leq 2\alpha\|u_2\|_{\L^{\infty}}\|w\|_{\L^2}\|\partial_xw\|_{\L^2}\leq\nu\|w\|_{\H_0^1}^2+\frac{C\alpha^2}{\nu}\|u_2\|_{\H_0^1}^2\|w\|_{\L^2}^2.
\end{align}
Using \eqref{3p51} and  \eqref{2.11} in  \eqref{4.59}, we get  
\begin{align}\label{4.60}
&e^{-\rho(\s)}\|w(\s)\|_{\L^2}^2+\nu\int_0^{\s}e^{-\rho(s)}\|w(s)\|_{\H_0^1}^2d s\nonumber\\&\leq \|w(0)\|_{\L^2}^2-\int_0^{\t}\rho'(s)e^{-\rho(s)}\|w(s)\|_{\L^2}^2d s +2\beta((1+\gamma+\gamma^2)\int_0^{\s}e^{-\rho(s)}\|w(s)\|_{\L^2}^2ds\nonumber\\&\quad+\frac{C\alpha^2}{\nu}\int_0^{\s}\|u_2(s)\|_{\H_0^1}^2\|w(s)\|_{\L^2}^2ds  \nonumber\\&\quad+\int_0^{\s}e^{-\rho(s)}\|\wi\sigma(s)\|_{\mathcal{L}_{Q}}^2d s+2\int_0^{\s}e^{-\rho(s)}\left(\widetilde{\sigma}(s)dW(s),w(s)\right).
\end{align}
	Note that the final term in the right hand side of the inequality (\ref{4.60}) is a local martingale. Let us take expectation in (\ref{4.60}), and  use the Hypothesis \ref{hyp} (H.2)  to get 
	\begin{align}\label{4.62}
	&\E\left[e^{-\rho(\s)}\|w(\s)\|_{\L^2}^2\right]\nonumber\\&\leq \E\left[\|w(0)\|_{\L^2}^2\right]+[2\beta(1+\gamma+\gamma^2)+L]\E\left[\int_0^{\s}e^{-\rho(s)}\|w(s)\|_{\L^2}^2d s\right].
	\end{align}
	We apply Gronwall's inequality in (\ref{4.62}) to obtain 
	\begin{align}\label{4.63}
	&\E\left[e^{-\rho(\s)}\|w(\s)\|_{\L^2}^2\right]\leq \E\left[\|w(0)\|_{\L^2}^2\right]e^{[2\beta(1+\gamma+\gamma^2)+L]T}.
	\end{align}
	Thus the initial data  $u_1(0)=u_2(0)=u_0$ leads to $w(\s)=0$, $\mathbb{P}$-a.s. But the fact that $\tau_N\to T$, $\mathbb{P}$-a.s., gives $w(t)=0$ and hence $u_1(t) = u_2(t)$, for all $t \in[0, T ]$, $\mathbb{P}$-a.s., and  the uniqueness follows.
\end{proof}

\begin{remark}[Regularity]
	Let us now assume $u_0\in\L^2(\Omega;\H_0^1(\Omega))\cap\L^{2p}(\Omega;\L^2(\mathcal{O})),$ for $p\geq 4$, and there exist a positive
	constant $\widetilde{K}$ such that, for all $u\in \H_0^1(\mathcal{O})$ and $t\in[0,T]$,
	\begin{equation*}
	\|A^{1/2}\sigma(t, u)\|^{2}_{\mathcal{L}_{Q}} 
	\leq \widetilde{K}\left(1 +\|u\|_{\H_0^1}^{2}\right).
	\end{equation*}
		Let us now apply finite dimensional It\^o's formula to the process $\|A^{1/2}u_n(\cdot)\|_{\L^2}^2$ to obtain
	\begin{align}\label{3.57}
	&\|A^{1/2}u_n(t)\|_{\L^2}^2+2\nu\int_0^t\|A u_n(s)\|_{\H_0^1}^2d s\nonumber\\&=\|A^{1/2}u_0\|_{\L^2}^2-2\alpha\int_0^t(B_n(u_n(s)),Au_n(s))ds+2\beta\int_0^t(c_n(u_n(s)),Au_n(s))ds\nonumber\\&\quad+\int_0^t\Tr(A^{1/2}\sigma(s,u_n(s)) Q(A^{1/2}\sigma(s,u_n(s)))^*) ds+2\int_0^t(\sigma(s,u_n(s)) d W(s),Au_n(s)).
	\end{align} 
	We estimate $-2\alpha(B(u_n),Au_n)$ and $2\beta(c(u_n),Au_n)$ using H\"older's, Gagliardo-Nirenberg and Young's inequalities as 
	\begin{align}\label{3.58}
	-2\alpha(B(u_n),Au_n)&\leq 2\alpha\|u_n\partial_xu_n\|_{\L^2}\|Au_n\|_{\L^2}\leq 2\alpha\|u_n\|_{\L^6}\|\partial_xu_n\|_{\L^3}\|Au_n\|_{\L^2}\nonumber\\&\leq 2C\alpha\|u_n\|_{\L^6}^{3/2}\|Au_n\|_{\L^2}^{3/2}\leq\frac{\nu}{2}\|Au_n\|_{\L^2}^2+\frac{27C\alpha^4}{2\nu^3}\|u_n\|_{\L^6}^6,\\
	2\beta(c(u_n),Au_n)&=2\beta(1+\gamma)(u_n^2,Au_n)-\beta\gamma(u_n,Au_n)-\beta(u_n^3,Au_n)\nonumber\\&\leq-2\beta\gamma\|\partial_xu_n\|_{\L^2}^2-2\beta\|u_n\partial_xu_n\|_{\L^2}^2+2\beta(1+\gamma)\|u_n\|_{\L^4}^2\|Au_n\|_{\L^2}\nonumber\\&\leq- 2\beta\gamma\|\partial_xu_n\|_{\L^2}^2-2\beta\|u_n\partial_xu_n\|_{\L^2}^2+\frac{\nu}{2}\|Au_n\|_{\L^2}^2+\frac{2\beta^2(1+\gamma)^2}{\nu}\|u_n\|_{\L^4}^4. \label{3p59}
	\end{align}
	Thus, using \eqref{3.58}-\eqref{3p59} in \eqref{3.57}, taking supremum over time from $0$ to $T$, and then taking expectation, we get 
	\begin{align}\label{3p60}
	&	\mathbb{E}\left[\sup_{t\in[0,T]}\|u_n(t)\|_{\H_0^1}^2+\nu\int_0^T\|A u_n(t)\|_{\L^2}^2d t+2\beta\gamma\int_0^T\|u_n(t)\|_{\H_0^1}^2dt+2\beta\int_0^T\|u_n(t)\partial_xu_n(t)\|_{\L^2}^2dt\right]\nonumber\\&\leq\E\left[\|u_0\|_{\H_0^1}^2\right]+\frac{27C\alpha^4}{2\nu^3}\E\left[\int_0^T\|u_n(t)\|_{\L^6}^6dt\right]+\frac{2\beta^2(1+\gamma)^2}{\nu}\E\left[\int_0^T\|u_n(t)\|_{\L^4}^4dt\right]\nonumber\\&\quad+\E\left[\int_0^T\|A^{1/2}\sigma(t,u_n(t))\|_{\mathcal{L}_Q}^2dt\right]+2\E\left[\sup_{t\in[0,T]}\left|\int_0^t(A^{1/2}\sigma(s,u_n(s)) d W(s),A^{1/2}u_n(s))\right|\right].
	\end{align}
	Let us take the final term from the right hand side of the inequality (\ref{3p60}) and use Burkholder-Davis-Gundy,  H\"{o}lder's and Young's inequalities to get
	\begin{align}\label{3p61}
	&2\E\left[\sup_{t\in[0,T]}\left|\int_0^{t}\left(A^{1/2}\sigma(s,u_n(s))dW(s),A^{1/2}u_n(s)\right)\right|\right]\nonumber\\&\leq 2\sqrt{3}\E\left[\int_0^{T}\|A^{1/2}\sigma(t,u_n(t))\|_{\mathcal{L}_{Q}}^2\|A^{1/2}u_n(t)\|_{\L^2}^2d t\right]^{1/2}\nonumber\\&\leq 2 \sqrt{3}\E\left[\sup_{t\in[0,T]}\|A^{1/2}u_n(t)\|_{\L^2}\left(\int_0^{T}\|A^{1/2}\sigma(t,u_n(t))\|_{\mathcal{L}_{Q}}^2d t\right)^{1/2}\right]\nonumber\\&\leq \frac{1}{2} \E\Bigg[\sup_{t\in[0,T]}\|u_n(t)\|_{\H_0^1}^2\Bigg]+6\E\Bigg[\int_0^{T}\|A^{1/2}\sigma(t,u_n(t))\|^2_{\mathcal{L}_{Q}}d
	t\Bigg].
	\end{align}
	Substituting \eqref{3p61} in \eqref{3p60}, we obtain 
	\begin{align}\label{3p62}
	&	\mathbb{E}\left[\sup_{t\in[0,T]}\|u_n(t)\|_{\H_0^1}^2+2\nu\int_0^T\|A u_n(t)\|_{\L^2}^2d t+4\beta\gamma\int_0^T\|u_n(t)\|_{\H_0^1}^2dt+4\beta\int_0^T\|u_n(t)\partial_xu_n(t)\|_{\L^2}^2dt\right]\nonumber\\& \leq 2\E\left[\|u_0\|_{\H_0^1}^2\right]+\frac{27C\alpha^4}{\nu^3}\E\left[\int_0^T\|u_n(t)\|_{\L^6}^6dt\right]+\frac{4\beta^2(1+\gamma)^2}{\nu}\E\left[\int_0^T\|u_n(t)\|_{\L^4}^4dt\right]\nonumber\\&\quad+14\widetilde{K}T+14\widetilde{K}\E\left[\int_0^T\|u_n(t)\|_{\H_0^1}^2dt\right].
	\end{align}
	Applying Gronwall's inequality in \eqref{3p62} gives 
	\begin{align}\label{3p63}
	\mathbb{E}\left[\sup_{t\in[0,T]}\|u_n(t)\|_{\H_0^1}^2\right]&\leq \bigg\{2\E\left[\|u_0\|_{\H_0^1}^2\right]+14\widetilde{K}T+\frac{4\beta^2(1+\gamma)^2}{\nu}\E\left[\int_0^T\|u_n(t)\|_{\L^4}^4dt\right]\nonumber\\&\qquad+\frac{27C\alpha^4}{\nu^3}\E\left[\int_0^T\|u_n(t)\|_{\L^6}^6dt\right]\bigg\}e^{14\widetilde{K}T}.
	\end{align}
	Using Gagliardo-Nirenberg and H\"older's inequalities, one can easily get 
	\begin{align*}
	\E\left[\int_0^T\|u_n(t)\|_{\L^6}^6dt\right]\leq& C\E\left[\sup_{t\in[0,T]}\|u_n(t)\|_{\L^2}^4\int_0^T\|\partial_xu_n(t)\|_{\L^2}^2dt\right]\nonumber\\&\leq C\left\{\E\left[\sup_{t\in[0,T]}\|u_n(t)\|_{\L^2}^8\right]\right\}^{1/2}\left\{\E\left[\left(\int_0^T\|u_n(t)\|_{\H_0^1}^2\right)^2\right]\right\}^{1/2}<+\infty,
	\end{align*}
	whenever $u_0\in\L^{2p}(\Omega;\L^2(\mathcal{O}))$, for $p\geq 4$. Combining \eqref{3p62}-\eqref{3p62}, we obtain 
	\begin{align}
		\mathbb{E}\left[\sup_{t\in[0,T]}\|u_n(t)\|_{\H_0^1}^2+2\nu\int_0^T\|A u_n(t)\|_{\L^2}^2d t\right]\leq C\left(\E\left[\|u_0\|_{\H_0^1}^2\right], \E\left[\|u_0\|_{\L^2}^{2p},\beta,\gamma,\nu,\alpha,K,\widetilde{K},T\right]\right),
	\end{align}
for $p\geq 4$. Thus, using the Banach-Alaoglu theorem, we can extract a subsequence such that 
\begin{equation}
\left\{
\begin{aligned}
	u_n&\xrightarrow{w^*}u\ \text{ in }\ \L^{2}(\Omega;\L^{\infty}(0,T;\H_0^1(\Omega))),\\
	u_n&\xrightarrow{w}u\ \text{ in }\ \L^2(\Omega;\L^2(0,T;\D(A))). 
\end{aligned}
\right.
\end{equation}
 Since $u(\cdot)$ is the unique strong solution of the system \eqref{abstract}, we obtain 	the regularity of $u_n(\cdot)$ as 
	$$u\in\L^2(\Omega;\L^{\infty}(0,T;\H_0^1(\mathcal{O}))\cap\L^2(0,T;\H^2(\mathcal{O}))),$$ and one can prove that $u$ has a continuous modification in $\C([0,T];\H_0^1(\mathcal{O}))\cap\L^2(0,T;\H^2(\mathcal{O}))$, $\mathbb{P}$-a.s.
\end{remark}

\section{The inviscid limit}\label{sec5}\setcounter{equation}{0}  In this section, we  discuss the inviscid limit of the equation \eqref{abstract} as $\beta\to 0$. Let $u(\cdot)$ be the unique strong solution of the system \eqref{abstract}. We consider the following stochastic Brugers' equation: 
\begin{equation}\label{48}
\left\{
\begin{aligned}
dv(t)&=[- Av(t)-\alpha B(v(t))]dt+\sigma(t,v(t))dW(t), \ t\in(0,T),\\
v(0)&=u_0\in\L^{2p}(\Omega;\L^{2}(\mathcal{O})). 
\end{aligned}
\right.
\end{equation}
The existence and uniqueness of strong solution of the above system  can be established in a similar way as in section \ref{sec3} (see \cite{GDP} also). For $u_0\in\L^{2p}(\Omega;\L^2(\mathcal{O}))$, $p>2$, the unique strong solution of the system \eqref{48} satisfies the energy inequality:
\begin{align}\label{4.9}
&\E\left[\sup_{t\in[0,T]}\|v(t)\|_{\L^2}^2+4\nu\int_0^{T}\|v(t)\|_{\H_0^1}^2d t\right]\leq
(2\|u_0\|_{\L^2}^2+14KT)e^{28KT}.
\end{align}
Also $u(\cdot)$ has the regularity $$u\in\L^{2p}(\Omega;\L^{\infty}(0,T;\L^2(\mathcal{O})))\cap\L^{2}(\Omega;\L^2(0,T;\H_0^1(\mathcal{O}))),$$ with a continuous modification having $u\in\C([0,T];\L^2(\mathcal{O}))\cap\L^2(0,T;\H_0^1(\mathcal{O})),$ $\mathbb{P}$-a.s. 
\begin{proposition}\label{prop5.1}
		Let $u(\cdot)$ be the unique strong solution of the stochastic Brugers-Huxley equation \eqref{abstract}. As $\beta\to 0$, the strong solution $v(\cdot)$ of the system \eqref{abstract} tends to the strong solution of the stochastic Brugers equation \eqref{48}. 
\end{proposition}
\begin{proof}
Let us define $w=u-v$, then $w$ satisfies: 
\begin{equation}\label{49}
\left\{\begin{aligned}
dw(t)&=[- Aw(t)-\alpha (B(u(t))-B(v(t)))+\beta c(u(t))]dt\\&\quad +(\sigma(t,u(t))-\sigma(t,v(t)))dW(t), \ t\in(0,T),\\
w(0)&=0.
\end{aligned}\right.
\end{equation}
Applying infinite dimensional It\^o's formula to the process $e^{-\rho(\cdot)}\|w(\cdot)\|_{\L^2}$, we  find 
\begin{align}\label{51}
&e^{-\rho(t)}\|w(t)\|_{\L^2}^2+2\nu\int_0^te^{-\rho(s)}\|\partial_x w(s)\|_{\L^2}^2ds\nonumber\\&=\|w(0)\|_{\L^2}^2-\int_0^{t}\rho'(s)e^{-\rho(s)}\|w(s)\|_{\L^2}^2d s-2\alpha\int_0^te^{-\rho(s)}(B(u(s))-B(v(s)),w(s))ds \nonumber\\&\quad+2\beta\int_0^te^{-\rho(s)}(c(u(s)),w(s))ds+\int_0^{t}e^{-\rho(s)}\|\wi\sigma(s)\|_{\mathcal{L}_{Q}}^2d s+2\int_0^{t}e^{-\rho(s)}\left(\widetilde{\sigma}(s)dW(s),w(s)\right),
\end{align}
where $\widetilde{\sigma}(\cdot)=\sigma(\cdot,u(\cdot))-\sigma(\cdot,v(\cdot))$ and $\rho(\cdot)$ is defined in \eqref{3pp51}.
A calculation similar to \eqref{3p51} gives 
\begin{align}\label{4p5}
2\alpha|(B(u)-B(v),w)|\leq{\nu}\|w\|_{\H_0^1}^2+\frac{C\alpha^2}{\nu}\|u\|_{\H_0^1}^2\|w\|_{\L^2}^2.
\end{align}
Applying H\"older's and Young's inequalities, we estimate $2\beta(c(u),w)$ as 
\begin{align}\label{4p6}
2\beta|(c(u),w)|&\leq 2\beta|(1+\gamma)(u^2,w)-\gamma(u,w)-(u^3,w)|\nonumber\\&\leq 2\beta(1+\gamma)\|u\|_{\L^4}^2\|w\|_{\L^2}+\beta\gamma\|u\|_{\L^2}\|w\|_{\L^2}+\beta\|w\|_{\L^{\infty}}\|u\|_{\L^3}^3
\nonumber\\&\leq\frac{\beta}{2}\|w\|_{\L^2}^2+2\beta(1+\gamma)^2\|u\|_{\L^4}^4+\frac{\beta}{2}\|w\|_{\L^2}^2+\frac{\beta\gamma^2}{2}\|u\|_{\L^2}^2+\frac{\nu}{2}\|w\|_{\H_0^1}^2+\frac{C\beta^2}{2\nu}\|u\|_{\L^3}^6.
\end{align}
Combining  \eqref{4p5}  and \eqref{4p6}, substituting it in \eqref{51} and then taking expectation, the fact that the final term in the right hand side of the equality \eqref{51} is a martingale, we obtain 
\begin{align}\label{4p7}
&\E\left[e^{-\rho(t)}\|w(t)\|_{\L^2}^2+\nu\int_0^te^{-\rho(s)}\|\partial_x w(s)\|_{\L^2}^2ds\right]\nonumber\\&\leq  -\E\left[\int_0^{t}\rho'(s)e^{-\rho(s)}\|w(s)\|_{\L^2}^2d s\right]+\frac{C\alpha^2}{2\nu}\E\left[\int_0^te^{-\rho(s)}\|u(s)\|_{\H_0^1}^2\|w(s)\|_{\L^2}^2ds\right]\nonumber\\&\quad+\beta\E\left[\int_0^te^{-\rho(s)}\|w(s)\|_{\L^2}^2ds\right]+2\beta(1+\gamma)^2\E\left[\int_0^te^{-\rho(s)}\|u(s)\|_{\L^4}^4ds\right]\nonumber\\&\quad+\frac{\beta\gamma^2}{2}\mathbb{E}\left[\int_0^te^{-\rho(s)}\|u(s)\|_{\L^2}^2ds\right]+\frac{C\beta^2}{2\nu}\E\left[\int_0^te^{-\rho(s)}\|u(s)\|_{\L^2}^2\|u(s)\|_{\L^4}^4ds\right]\nonumber\\&\leq \beta\E\left[\int_0^te^{-\rho(s)}\|w(s)\|_{\L^2}^2ds\right]+2\beta(1+\gamma)^2\E\left[\int_0^T\|u(t)\|_{\L^4}^4dt\right]\nonumber\\&\quad+\frac{\beta\gamma^2}{2}\mathbb{E}\left[\int_0^T\|u(t)\|_{\L^2}^2dt\right]+\frac{C\beta^2}{2\nu}\left\{\E\left[\sup_{t\in[0,T]}\|u(t)\|_{\L^2}^4\right]\right\}^{1/2}\left\{\E\left[\int_0^T\|u(t)\|_{\L^4}^4dt\right]^2\right\}^{1/2},
\end{align}
where we used H\"older's inequality. Note that the final term from the right hand side of the inequality is bounded by using \eqref{energy2} and \eqref{3p27}. An application of Gronwall's inequality in \eqref{4p7} yields 
\begin{align}\label{4p8}
\E\left[e^{-\rho(t)}\|w(t)\|_{\L^2}^2\right]&\leq\beta \bigg\{2(1+\gamma)^2\E\left[\int_0^T\|u(t)\|_{\L^4}^4dt\right]+\frac{\gamma^2}{2}\mathbb{E}\left[\int_0^T\|u(t)\|_{\L^2}^2dt\right]\nonumber\\&\quad+\frac{C\beta}{2\nu}\left\{\E\left[\sup_{t\in[0,T]}\|u(t)\|_{\L^2}^4\right]\right\}^{1/2}\left\{\E\left[\int_0^T\|u(t)\|_{\L^4}^4dt\right]^2\right\}^{1/2}\bigg\}e^{\beta t},
\end{align} 
for all $t\in[0,T]$. Passing $\beta\to 0$ in \eqref{4p8}, we find $u(t)\to v(t)$, for all $t\in[0,T]$, $\mathbb{P}$-a.s.
\end{proof}

Let us now discuss the inviscid limit of the equation \eqref{abstract} as $\alpha\to 0$. We consider the following Huxley equation, for $(x,t)\in\Omega\times(0,T)$: 
\begin{equation}\label{514}
\left\{
\begin{aligned}
dz(t)&=[- \nu Az(t)+\beta c(z(t))]dt+\sigma(t,z(t))dW(t), \ t\in(0,T),\\
z(0)&=u_0\in\L^{2p}(\Omega;\L^{2}(\mathcal{O})). 
\end{aligned}\right.
\end{equation}
From \eqref{3.7}, for $F(\cdot)=\nu A+\beta c(\cdot)$,  it can be easily seen that 
	\begin{align}
&\langle F(u)-F(v),u-v\rangle +\beta(1+\gamma+\gamma^2)\|u-v\|_{\L^2}^2\geq 0,
\end{align}
and hence the operator $F+\lambda I$ is monotone, where $\lambda=\beta(1+\gamma+\gamma^2)$. Since the operator  $F+\lambda I:\H_0^1(\mathcal{O})\to\H^{-1}(\mathcal{O})$ is monotone and hemicontinuous, using Theorem Theorem 1.3, Chapter 2,\cite{VB}, the operator $F+\lambda I$ is a maximal monotone operator., Moreover, Corollary 1.2, Chapter 2,\cite{VB} gives us that $R(F+\lambda I)=\H^{-1}(\mathcal{O})$. 
The existence and uniqueness of strong solution $$z\in\L^{2p}(\Omega;\L^{\infty}(0,T;\L^2(\Omega)))\cap\L^{2}(\Omega;\L^2(0,T;\H_0^1(\Omega)))\cap\L^4(\Omega;\L^4(0,T;\L^4(\mathcal{O})))$$ to the system \eqref{514} can be proved in a similar way as in Theorem \ref{exis}. Moreover, $z(\cdot)$ satisfies:
\begin{align}\label{5.15} 
	&\E\left[\sup_{t\in[0,T]}\|z(t)\|_{\L^2}^2+4\nu\int_0^{T}\|z(t)\|_{\H_0^1}^2d t+2\beta\int_0^{T}\|z(t)\|_{\L^4}^4dt\right]\nonumber\\&\quad \leq
(2\|u_0\|_{\L^2}^2+14KT)e^{4(\beta(1+\gamma^2)+7K)T}.
\end{align}
Then, we have the following result: 
\begin{proposition}\label{prop4.2}
	Let $u(\cdot)$ be the unique strong solution of the stochastic Brugers-Huxley equation \eqref{abstract}. As $\alpha\to 0$, the strong solution $v(\cdot)$ of the system \eqref{abstract} tends to the strong solution of the stochastic Huxley equation \eqref{514}. 
\end{proposition}
\begin{proof}
	Let us define $w=u-z$. Then $w$ satisfies: 
	\begin{equation}\label{515}
	\left\{\begin{aligned}
dw(t)&=[- Aw(t)+\beta (c(u(t))-c(z(t)))-\alpha B(u(t))]dt\\&\quad +[\sigma(t,u(t))-\sigma(t,z(t))]dW(t), \ t\in(0,T),\\
z(0)&=0.
	\end{aligned}\right.
	\end{equation}
Applying infinite dimensional It\^o's formula to the process $\|w(\cdot)\|_{\L^2}^2$, we find 
	\begin{align}\label{516}
	&e^{-\widehat{\rho}(t)}\|w(t)\|_{\L^2}^2+2\nu\int_0^te^{-\widehat{\rho}(s)}\|\partial_x w(s)\|_{\L^2}^2ds\nonumber\\&=\|w(0)\|_{\L^2}^2-\int_0^{t}\widehat{\rho}'(s)e^{-\widehat{\rho}(s)}\|w(s)\|_{\L^2}^2d s+2\beta\int_0^te^{-\widehat{\rho}(s)}(c(u(s))-c(z(s)),w(s))ds\nonumber\\&\quad-2\alpha\int_0^te^{-\widehat{\rho}(s)}(B(u(s)),w(s))ds
+\int_0^{t}e^{-\widehat{\rho}(s)}\|\widehat\sigma(s)\|_{\mathcal{L}_{Q}}^2d s+2\int_0^{t}e^{-\widehat{\rho}(s)}\left(\widehat{\sigma}(s)dW(s),w(s)\right),
\end{align}
where $\widehat{\sigma}(\cdot)=\sigma(\cdot,u(\cdot))-\sigma(\cdot,z(\cdot))$ and 
\begin{align*}
\widehat{\rho}(t)=C\alpha \int_0^t\|u(s)\|_{\H_0^1}^2ds, \ \text{
so that }\ 
\widehat{\rho}'(t)=C\alpha\|u(t)\|_{\H_0^1}^2, \ \text{ a.e.}
\end{align*}
	The above equality implies 
	\begin{align}\label{517}
	&e^{-\widehat{\rho}(t)}\|w(t)\|_{\L^2}^2+2\nu\int_0^te^{-\widehat{\rho}(s)}\| w(s)\|_{\H_0^1}^2ds\nonumber\\&=\|w(0)\|_{\L^2}^2-\int_0^{t}\widehat{\rho}'(s)e^{-\widehat{\rho}(s)}\|w(s)\|_{\L^2}^2d s-2\alpha\int_0^te^{-\widehat{\rho}(s)}(u(s)\partial_xw(s),w(s))ds \nonumber\\&\quad+2\beta(1+\gamma+\gamma^2)\int_0^te^{-\widehat{\rho}(s)}\|w(s)\|_{\L^2}^2ds
+\int_0^{t}e^{-\widehat{\rho}(s)}\|\widehat\sigma(s)\|_{\mathcal{L}_{Q}}^2d s\nonumber\\&\quad+2\int_0^{t}e^{-\widehat{\rho}(s)}\left(\widehat{\sigma}(s)dW(s),w(s)\right).
	\end{align}
	We estimate the term $-2\alpha(u\partial_xw,w)$ using H\"older's and Young's inequalities as 
	\begin{align}\label{518}
-2\alpha(u\partial_xw,w)&\leq 2\alpha\|u\|_{\L^{\infty}}\|\partial_xu\|_{\L^2}\|w\|_{\L^2}\leq {C\alpha}\|u\|_{\H_0^1}^2(1+\|w\|_{\L^2}^2).
	\end{align}   
	Using  \eqref{518} in \eqref{517} and then taking expectation, we obtain  
	\begin{align}\label{522}
		&\E\bigg[e^{-\widehat{\rho}(t)}\|w(t)\|_{\L^2}^2+2\nu\int_0^te^{-\widehat{\rho}(s)}\| w(s)\|_{\H_0^1}^2ds\bigg]\nonumber\\&\leq C\alpha\E\left[\int_0^{T}\|u(t)\|_{\H_0^1}^2dt\right]+ 2\beta(1+\gamma+\gamma^2)\E\left[\int_0^{t}e^{-\widehat{\rho}(s)}\|w(s)\|_{\L^2}^2ds\right]\nonumber\\&\quad+\E\left[\int_0^{t}e^{-\widehat{\rho}(s)}\|\widehat\sigma(s)\|_{\mathcal{L}_{Q}}^2d s\right]\nonumber\\&\leq C\alpha\E\left[\int_0^{T}\|u(t)\|_{\H_0^1}^2dt\right]+[2\beta(1+\gamma+\gamma^2)+L]\E\left[\int_0^{t}e^{-\widehat{\rho}(s)}\|w(s)\|_{\L^2}^2ds\right],
	\end{align}
where we used Hypothesis \ref{hyp} (H.3). 	An application of Gronwall's inequality in \eqref{522} gives 
	\begin{align}\label{524}
	&\E\left[e^{-\widehat{\rho}(t)}\|w(t)\|_{\L^2}^2\right]\leq C\alpha\E\left[\int_0^{T}\|u(t)\|_{\H_0^1}^2dt\right]e^{[2\beta(1+\gamma+\gamma^2)+L]T}.
	\end{align}
Passing $\alpha\to 0$ in \eqref{524}, one can easily obtain that $u(t)\to z(t)$, for all $t\in[0,T]$, $\mathbb{P}$-a.s.
\end{proof}

\section{Large Deviations Principle and Exit Time Estimates}\label{sec6}\setcounter{equation}{0} 
In this section, we examine the small noise asymptotic granted  by large deviations theory and use it to estimate the exit time estimates. We take the initial data $u_0\in\L^2(\mathcal{O})$ as deterministic.  Let us first provide some basics definitions of  large deviations theory.
\subsection{Preliminaries}
Let us denote by $\mathscr{E}$, a complete separable metric space  (Polish space)  with the Borel $\sigma$-field $\mathscr{B}(\mathscr{E})$.

\begin{definition}
	A function $\mathrm{I} : \mathscr{E}\rightarrow [0, \infty]$ is called a \emph{rate
		function} if $\I$ is lower semicontinuous. A rate function $\I$ is
	called a \emph{good rate function} if for arbitrary $M \in [0,
	\infty)$, the level set $K_M = \big\{x\in\mathscr{E}: \I(x)\leq M\big\}$ is compact in $\mathscr{E}$.
\end{definition}
\begin{definition}[Large deviation principle]\label{LDP}\label{def4.2} Let $\I$ be a rate function on $\mathscr{E}$. A family $\big\{\mathrm{X}^{\varepsilon}: \varepsilon
	> 0\big\}$ of $\mathscr{E}$-valued random elements is said to satisfy \emph{the large deviation principle} on $\mathscr{E}$ with rate function $\I$, if the following two conditions hold:
	\begin{enumerate}
		\item[(i)] (Large deviation upper bound) For each closed set $\F\subset \mathscr{E}$:
		$$ \limsup_{\varepsilon\rightarrow 0} \varepsilon\log \mathbb{P}\left(\mathrm{X}^{\e}\in\F\right) \leq -\inf_{x\in \F} \I(x),$$
		\item[(ii)] (Large deviation lower bound) for each open set $\mathrm{G}\subset \mathscr{E}$:
		$$ \liminf_{\varepsilon\rightarrow 0}\varepsilon \log \mathbb{P}(\mathrm{X}^{\e}\in\mathrm{G}) \geq -\inf_{x\in \mathrm{G}} \I(x).$$
	\end{enumerate}
\end{definition}

\begin{definition}
	Let $\I$ be a rate function on $\mathscr{E}$. A family $\big\{\mathrm{X}^{\e} :\e > 0\big\}$ of $\mathscr{E}$-valued random elements is said to satisfy the \emph{Laplace principle} on $\mathscr{E}$ with rate function $\I$ if for each real-valued, bounded and continuous function $h$ defined on $\mathscr{E}$, i.e., for $h\in\C_b(\mathscr{E})$,
	\begin{equation}\label{LP}
	\lim_{\varepsilon \rightarrow 0} {\varepsilon }\log
	\mathbb{E}\left\{\exp\left[-
	\frac{1}{\varepsilon}h(\mathrm{X}^{\varepsilon})\right]\right\} = -\inf_{x
		\in \mathscr{E}} \big\{h(x) + \I(x)\big\}. 
	\end{equation}
\end{definition}

\begin{lemma}[Varadhan's Lemma \cite{Va}]\label{VL}
	Let $\mathscr{E}$ be a Polish space and $\{\mathrm{X}^{\varepsilon}: \varepsilon > 0\}$
	be a family of $\mathscr{E}$-valued random elements satisfying LDP with rate
	function $\I$. Then $\{\mathrm{X}^{\varepsilon}: \varepsilon > 0\}$ satisfies
	the Laplace principle on $\mathscr{E}$ with the same rate function $\I$.
\end{lemma}
\begin{lemma}[Bryc's Lemma \cite{DZ}]\label{BL}
	The Laplace principle implies the LDP with the same rate function.
\end{lemma}
Note that, Varadhan's Lemma together with Bryc's converse of
Varadhan's Lemma state that for Polish space valued random elements,
the Laplace principle and the large deviation principle are
equivalent. The LDP for the 2D stochastic Navier-Stokes equations is established in \cite{SSSP}.

Next Theorem shows that the LDP is preserved under continuous mappings, and is known as \emph{contraction principle}. 

\begin{theorem}[Contraction principle, Theorem 4.2.1, \cite{DZ}]\label{thm4.6}
	Let $\mathscr{E}$ and $\mathscr{G}$ be Hausdorff
	topological spaces and $f : \mathscr{E}\to \mathscr{G}$ a continuous function. Let us consider a good rate function $I : \mathscr{E}\to  [0,\infty]$. 
	\begin{enumerate}
		\item [(a)] For each $y \in\mathscr{E}'$, define
		\begin{align}
		\mathrm{J}(y):=\inf\left\{\mathrm{I}(x):x\in\mathscr{E},y=f(x)\right\}.
		\end{align}
		\item [(b)] If $\mathrm{I}$ controls the LDP associated with a family of probability measures
		$\{\mu_{\e}\}$ on $\mathscr{E}$, then $\mathrm{I}$ controls the LDP associated with the family of probability measures $\{\mu_{\e}\circ f^{-1}\}$ on $\mathscr{G}$. 
	\end{enumerate}
\end{theorem}
Let us consider the following stochastic heat equation (see \cite{DaZ}):
\begin{equation}\label{7p1}
\left\{
\begin{aligned}
dz(t)+ \nu Az(t)dt&=dW(t), \ t\in(0,T),\\
z(0)&=0. 
\end{aligned}
\right.
\end{equation}
For $\Tr(Q)<\infty$, one can show that there exists a unique pathwise strong solution to the system \eqref{7p1} with trajectories in $\C([0,T];\L^2(\mathcal{O}))\cap\L^2(0,T;\H_0^1(\mathcal{O}))$, $\mathbb{P}$-a.s. and satisfies the following energy estimate:
\begin{align}
\mathbb{E}\left[\sup_{t\in[0,T]}\|z(t)\|_{\L^2}^2+\nu\int_0^T\|z(s)\|_{\H_0^1}^2ds\right]\leq CT\Tr(Q). 
\end{align}
Let us define $v=u-z$ and consider the system satisfied by $v$ as 
\begin{equation}\label{7.5}
\left\{
\begin{aligned}
dv(t)+\nu Av(t)dt&=-\alpha B(v(t)+z(t))dt+\beta c(v(t)+z(t))dt, \ t\in(0,T),\\
v(0)&=u_0\in\L^{2}(\mathcal{O}). 
\end{aligned}
\right.
\end{equation}
Note that the randomness in the system \eqref{7.5} comes through $z(\cdot,\omega)$ and the system can be solved for each $\omega\in\Omega$ and one can show that the trajectories belong to $\C([0, T ]; \L^2(\mathcal{O}))\cap\L^2(0, T ; \H^1_0(\mathcal{O}))$, $\mathbb{P}$-a.s. Taking inner product with $v(\cdot)$ with the first equation in \eqref{7.5}, we find 
\begin{align}\label{7.6}
&\frac{1}{2}\frac{d}{dt}\|v(t)\|_{\L^2}^2+\nu\|\partial_xv(t)\|_{\L^2}^2=-\alpha(B(v(t)+z(t)),v(t))+\beta(c(v(t)+z(t)),v(t)).
\end{align}
Using an integration by parts, \eqref{6}, H\"older's and Young's inequalities, it can be easily deduced that 
\begin{align}\label{7.7}
-\alpha(B(v+z),v)&=-\alpha(v\partial_xv,v)-\alpha(z\partial_xv,v)-\alpha(v\partial_xz,v)-\alpha(z\partial_xz,v)\nonumber\\&=2\alpha(z,v\partial_xv)-\alpha(z\partial_xz,v)\nonumber\\&\leq 2\alpha\|z\|_{\L^{\infty}}\|v\|_{\L^2}\|\partial_xv\|_{\L^2}+\alpha\|z\|_{\L^{\infty}}\|\partial_xz\|_{\L^2}\|v\|_{\L^2}\nonumber\\&\leq\frac{\nu}{2}\|\partial_xv\|_{\L^2}^2+\frac{2C\alpha}{\nu}\|z\|_{\H_0^1}^2\|v\|_{\L^2}^2+C\alpha(1+\|v\|_{\L^2}^2)\|z\|_{\H_0^1}^2. 
\end{align}
Similarly, we estimate $\beta(c(v+z),v)$ as 
\begin{align}\label{7.8}
\beta(c(v+z),v)&=-\beta\gamma\|v\|_{\L^2}^2-\beta\|v\|_{\L^4}^4+\beta(1+\gamma)(v^2,v)+2\beta(1+\gamma)(vz,v)\nonumber\\&\quad+\beta(1+\gamma)(z^2,v)-\beta\gamma(z,v)-3\beta(v^2z,v)-3\beta(vz^2,v)-\beta(z^3,v). 
\end{align}
Let us compute each term appearing the right hand side of the equality \eqref{7.8} as 
\begin{align}
\beta(1+\gamma)|(v^2,v)|&\leq\beta(1+\gamma)\|v\|_{\L^2}\|v\|_{\L^4}^2\leq \frac{\beta}{4}\|v\|_{\L^4}^4+\beta(1+\gamma)^2\|v\|_{\L^2}^2,\\
2\beta(1+\gamma)|(vz,v)|&\leq 2\beta(1+\gamma)\|z\|_{\L^{\infty}}\|v\|_{\L^2}^2\leq C\beta(1+\gamma)(1+\|z\|_{\H_0^1}^2)\|v\|_{\L^2}^2,\\
\beta(1+\gamma)|(z^2,v)|&\leq\beta(1+\gamma)\|z\|_{\L^2}\|z\|_{\L^{\infty}}\|v\|_{\L^2}\leq C\beta(1+\gamma)\|z\|_{\H_0^1}^2(1+\|v\|_{\L^2}^2),\\
\beta\gamma|(z,v)|&\leq\beta\gamma\|z\|_{\L^2}\|v\|_{\L^2}\leq\frac{\beta\gamma}{2}\|z\|_{\L^2}^2+\frac{\beta\gamma}{2}\|v\|_{\L^2}^2,\\
3\beta|(v^2z,v)|&\leq\|z\|_{\L^{\infty}}\|v\|_{\L^4}^2\|v\|_{\L^2}\leq\frac{\beta}{4}\|v\|_{\L^4}^4+9C\beta\|z\|_{\H_0^1}^2\|v\|_{\L^2}^2,\\
3\beta|(vz^2,v)|&\leq 3\beta\|z\|_{\L^{\infty}}^2\|v\|_{\L^2}^2\leq 3C\beta\|z\|_{\H_0^1}^2\|v\|_{\L^2}^2,\\
\beta|(z^3,v)|&\leq\beta\|z\|_{\L^6}^3\|v\|_{\L^2}\leq C\beta\|z\|_{\H_0^1}\|z\|_{\L^2}^2\|v\|_{\L^2}\leq\beta\|z\|_{\H_0^1}^2+\frac{C\beta}{4}\|z\|_{\L^2}^4\|v\|_{\L^2}^2. \label{7.15}
\end{align}
Combining \eqref{7.7}-\eqref{7.15},  substituting it in \eqref{7.6} and then integrating from $0$ to $t$, we get 
\begin{align}\label{7.16}
&\|v(t)\|_{\L^2}^2+\nu\int_0^t\|\partial_xv(s)\|_{\L^2}^2ds+\beta\int_0^t\|v(s)\|_{\L^4}^4ds\nonumber\\&\leq\|v_0\|_{\L^2}^2+C(\alpha+\beta)\int_0^t\|z(s)\|_{\H_0^1}^2ds+\frac{\beta\gamma}{2}\int_0^t\|z(s)\|_{\L^2}^2ds\\&\quad +C\int_0^t\bigg\{\left[\alpha\left(\frac{1}{\nu}+1\right)+\beta(1+\gamma)\right]\|z(s)\|_{\H_0^1}^2+\beta\|z(s)\|_{\L^2}^4+\beta(1+\gamma+\gamma^2)\bigg\}\|v(s)\|_{\L^2}^2ds,\nonumber
\end{align}
for all $0\leq t\leq T$. An application of Gronwall's inequality in \eqref{7.16} yields 
\begin{align}\label{7.17}
&\|v(t)\|_{\L^2}^2+\nu\int_0^t\|\partial_xv(s)\|_{\L^2}^2ds+\beta\int_0^t\|v(s)\|_{\L^4}^4ds\nonumber\\&\leq\left\{\|u_0\|_{\L^2}^2+C\left(\alpha+\beta\right)\int_0^T\|z(t)\|_{\H_0^1}^2dt\right\}\nonumber\\&\quad\times\exp\bigg\{C\left[\alpha\left(\frac{1}{\nu}+1\right)+\beta(1+\gamma)+\beta\sup_{t\in[0,T]}\|z(t)\|_{\L^2}^2\right]\int_0^T\|z(t)\|_{\H_0^1}^2dt+C\left[\beta(1+\gamma+\gamma^2)\right]t\bigg\},
\end{align}
for all $t\in[0,T]$. 
\begin{lemma}\label{lem7.7}
	Let a function $\psi\in\mathscr{E}:=\C([0,T];\L^2(\mathcal{O}))\cap\L^2(0,T;\H_0^1(\mathcal{O}))$ be given. Let the map \begin{align}\label{718}\Psi:\psi\mapsto v_{\psi}\end{align} be defined by 
	\begin{equation}\label{7.18}
	\left\{
	\begin{aligned}
	dv_{\psi}(t)+\nu Av_{\psi}(t)dt&=-\alpha B(v_{\psi}(t)+\psi(t))dt+\beta c(v_{\psi}(t)+\psi(t))dt, \ t\in(0,T),\\
	v_{\psi}(0)&=u_0\in\L^{2}(\mathcal{O}). 
	\end{aligned}
	\right.
	\end{equation}
	Then the map $\Psi$ is continuous from the space $\mathscr{E}$ to $\mathscr{E}$. 
\end{lemma}
\begin{proof}
	Let us consider two functions $\psi_1$ and $\psi_2$ in $\mathscr{E}$. We denote the corresponding solutions of \eqref{7.18} as $v_{\psi_i}$, for $i=1,2$. Then $v_{\psi}=v_{\psi_1}-v_{\psi_2}$, where $\psi=\psi_1-\psi_2$, satisfies: 
		\begin{equation}\label{7.19}
	\left\{
	\begin{aligned}
	dv_{\psi}(t)+\nu Av_{\psi}(t)dt&=-\alpha \left[B(v_{\psi_1}(t)+\psi_1(t))-B(v_{\psi_2}(t)+\psi_2(t))\right]dt\\&\quad+\beta \left[c(v_{\psi_1}(t)+\psi_1(t))-c(v_{\psi_2}(t)+\psi_2(t))\right]dt, \ t\in(0,T),\\
	v_{\psi}(0)&=0. 
	\end{aligned}
	\right.
	\end{equation}
	Taking inner product with $v_{\psi}$ to the first equation in \eqref{7.19} to find 
	\begin{align}\label{7.20}
	\frac{1}{2}\frac{d}{dt}\|v_{\psi}(t)\|_{\L^2}^2+\nu\|\partial_xv_{\psi}(t)\|_{\L^2}^2&= -\alpha(B(v_{\psi_1}(t)+\psi_1(t))-B(v_{\psi_2}(t)+\psi_2(t)),v_{\psi})\nonumber\\&\quad+\beta (c(v_{\psi_1}(t)+\psi_1(t))-c(v_{\psi_2}(t)+\psi_2(t)),v_{\psi}).
	\end{align}
	We estimate $-\alpha(B(v_{\psi_1}+\psi_1)-B(v_{\psi_2}+\psi_2),v_{\psi})$ as 
	\begin{align}\label{7.21}
	-\alpha(B(v_{\psi_1}+\psi_1)-B(v_{\psi_2}+\psi_2),v_{\psi})&=-\alpha(v_{\psi}\partial_x(v_{\psi_1}+\psi_1),v_{\psi})-\alpha(\psi\partial_x(v_{\psi_1}+\psi_1),v_{\psi})\nonumber\\&\quad-\alpha((v_{\psi_2}+\psi_2)\partial_xv_{\psi},v_{\psi})-\alpha((v_{\psi_2}+\psi_2)\partial_x\psi,v_{\psi})\nonumber\\&=:\sum_{i=1}^4I_i,
	\end{align}
	where $I_i$, for $i=1,\ldots,4$ are the terms appearing in the right hand side of the equality \eqref{7.21}. We estimate $I_i$, for $i=1,\ldots,4$ using H\"older's and Young's inequalities as 
	\begin{align}
|I_1|&=|\alpha(v_{\psi_1}+\psi_1,v_{\psi}\partial_xv_{\psi})|\leq\alpha\|v_{\psi_1}+\psi_1\|_{\L^{\infty}}\|v_{\psi}\|_{\L^2}\|\partial_xv_{\psi}\|_{\L^2}\nonumber\\&\leq\frac{\nu}{4}\|\partial_xv_{\psi}\|_{\L^2}^2+\frac{C}{\nu}\left(\|v_{\psi_1}\|_{\H_0^1}^2+\|\psi_1\|_{\H_0^1}^2\right)\|v_{\psi}\|_{\L^2}^2,\\
|I_2|&\leq \alpha\|\psi\|_{\L^{\infty}}\|\partial_x(v_{\psi_1}+\psi_1)\|_{\L^2}\|v_{\psi}\|_{\L^2}\leq C\|\psi\|_{\H_0^1}^2+\frac{\alpha^2}{2}\left(\|v_{\psi_1}\|_{\H_0^1}^2+\|\psi_1\|_{\H_0^1}^2\right)\|v_{\psi}\|_{\L^2}^2,\\
|I_3|&\leq\frac{\nu}{4}\|\partial_xv_{\psi}\|_{\L^2}^2+\frac{C}{\nu}\left(\|v_{\psi_2}\|_{\H_0^1}^2+\|\psi_2\|_{\H_0^1}^2\right)\|v_{\psi}\|_{\L^2}^2,\\
|I_4|&\leq\alpha\|v_{\psi_2}+\psi_2\|_{\L^{\infty}}\|\partial_x\psi\|_{\L^2}\|v_{\psi}\|_{\L^2}\leq\|\psi\|_{\H_0^1}^2+\frac{C\alpha^2}{2}\left(\|v_{\psi_2}\|_{\H_0^1}^2+\|\psi_2\|_{\H_0^1}^2\right)\|v_{\psi}\|_{\L^2}^2.
	\end{align}
	We estimate $\beta (c(v_{\psi_1}+\psi_1)-c(v_{\psi_2}+\psi_2),v_{\psi})$ as 
	\begin{align}\label{7.26}
	&\beta (c(v_{\psi_1}+\psi_1)-c(v_{\psi_2}+\psi_2),v_{\psi})\nonumber\\&= \beta(((v_{\psi_1}+\psi_1)-(v_{\psi_2}+\psi_2))[(1+\gamma)(v_{\psi_1}+\psi_1+v_{\psi_2}+\psi_2)\nonumber\\&\qquad-(\gamma+(v_{\psi_1}+\psi_1)^2+(v_{\psi_1}+\psi_1)(v_{\psi_1}+\psi_1)+(v_{\psi_2}+\psi_2)^2)],v_{\psi})\nonumber\\&=\beta(v_{\psi}[(1+\gamma)(v_{\psi_1}+\psi_1+v_{\psi_2}+\psi_2)\nonumber\\&\qquad-(\gamma+(v_{\psi_1}+\psi_1)^2+(v_{\psi_1}+\psi_1)(v_{\psi_1}+\psi_1))+(v_{\psi_2}+\psi_2)^2],v_{\psi})\nonumber\\&\quad+\beta(\psi[(1+\gamma)(v_{\psi_1}+\psi_1+v_{\psi_2}+\psi_2)\nonumber\\&\qquad-(\gamma+(v_{\psi_1}+\psi_1)^2+(v_{\psi_1}+\psi_1)(v_{\psi_1}+\psi_1)+(v_{\psi_2}+\psi_2)^2)],v_{\psi})\nonumber\\&=-\beta\gamma\|v_{\psi}\|_{\L^2}^2+\beta(1+\gamma)(v_{\psi}(v_{\psi_1}+\psi_1+v_{\psi_2}+\psi_2),v_{\psi})\nonumber\\&\quad-\beta(v_{\psi}((v_{\psi_1}+\psi_1)^2+(v_{\psi_1}+\psi_1)(v_{\psi_1}+\psi_1)+(v_{\psi_2}+\psi_2)^2),v_{\psi})\nonumber\\&\quad+\beta(1+\gamma)(\psi(v_{\psi_1}+\psi_1+v_{\psi_2}+\psi_2),v_{\psi})-\beta\gamma(\psi,v_{\psi})\nonumber\\&\quad-\beta(\psi((v_{\psi_1}+\psi_1)^2+(v_{\psi_1}+\psi_1)(v_{\psi_1}+\psi_1)+(v_{\psi_2}+\psi_2)^2),v_{\psi})=:\sum_{i=5}^9I_i,
	\end{align}
	where $I_i$, for $i=5,\ldots,8$ are the final five terms appearing in the right hand side of the equality \eqref{7.26}. We estimate $I_i$, for $i=5,\ldots,8$ using H\"older's and Young's inequalities as
	\begin{align}
	|I_5|&\leq\beta(1+\gamma)\|v_{\psi_1}+\psi_1+v_{\psi_2}+\psi_2\|_{\L^{\infty}}\|v_{\psi}\|_{\L^2}^2\nonumber\\&\leq C\beta(1+\gamma)\left(\|v_{\psi_1}\|_{\H_0^1}+\|\psi_1\|_{\H_0^1}+\|v_{\psi_2}\|_{\H_0^1}+\|\psi_2\|_{\H_0^1}\right)\|v_{\psi}\|_{\L^2}^2,\\
	|I_6|&\leq\beta\|(v_{\psi_1}+\psi_1)^2+(v_{\psi_1}+\psi_1)(v_{\psi_1}+\psi_1)+(v_{\psi_2}+\psi_2)^2\|_{\L^{\infty}}\|v_{\psi}\|_{\L^2}^2\nonumber\\&\leq C\beta\left(\|v_{\psi_1}\|_{\H_0^1}^2+\|\psi_1\|_{\H_0^1}^2+\|v_{\psi_2}\|_{\H_0^1}^2+\|\psi_2\|_{\H_0^1}^2\right)\|v_{\psi}\|_{\L^2}^2,\\
	|I_7|&\leq \beta(1+\gamma)\|\psi\|_{\L^2}\|v_{\psi_1}+\psi_1+v_{\psi_2}+\psi_2\|_{\L^{\infty}}\|v_{\psi}\|_{\L^2}\nonumber\\&\leq\frac{\beta(1+\gamma)}{2}\|\psi\|_{\L^2}^2+C\beta(1+\gamma)\left(\|v_{\psi_1}\|_{\H_0^1}^2+\|\psi_1\|_{\H_0^1}^2+\|v_{\psi_2}\|_{\H_0^1}^2+\|\psi_2\|_{\H_0^1}^2\right)\|v_{\psi}\|_{\L^2}^2,\\ |I_8|&\leq\beta\gamma\|\psi\|_{\L^2}\|v_{\psi}\|_{\L^2}\leq\frac{\beta\gamma}{2}\|\psi\|_{\L^2}^2+\frac{\beta\gamma}{2}\|v_{\psi}\|_{\L^2}^2,\\
	|I_9|&\leq\beta\|\psi\|_{\L^2}\|(v_{\psi_1}+\psi_1)^2+(v_{\psi_1}+\psi_1)(v_{\psi_1}+\psi_1)+(v_{\psi_2}+\psi_2)^2\|_{\L^{\infty}}\|v_{\psi}\|_{\L^2}\nonumber\\&\leq\frac{\beta}{2}\|\psi\|_{\L^2}^2+C\beta\left(\|v_{\psi_1}\|_{\H_0^1}^2+\|\psi_1\|_{\H_0^1}^2+\|v_{\psi_2}\|_{\H_0^1}^2+\|\psi_2\|_{\H_0^1}^2\right)\|v_{\psi}\|_{\L^2}^2.\label{7.31}
	\end{align}
	Combining \eqref{7.21}-\eqref{7.31}, substituting it in \eqref{7.20} and then integrating it from $0$ to $t$ to get
	\begin{align}\label{7.32}
&\|v_{\psi}(t)\|_{\L^2}^2+\nu\int_0^t\|\partial_xv_{\psi}(s)\|_{\L^2}^2ds+2\beta\gamma\int_0^t\|v_{\psi}(s)\|_{\L^2}^2ds\nonumber\\&\leq C\left(\frac{1}{\nu}+\alpha^2+\beta(1+\gamma)\right)\int_0^t\left(\|v_{\psi_1}(s)\|_{\H_0^1}^2+\|\psi_1(s)\|_{\H_0^1}^2+\|v_{\psi_2}(s)\|_{\H_0^1}^2+\|\psi_2(s)\|_{\H_0^1}^2\right)\|v_{\psi}(s)\|_{\L^2}^2ds\nonumber\\&\quad+C\int_0^t\|\psi(s)\|_{\H_0^1}^2ds+\frac{\beta}{2}(1+\gamma)\int_0^t\|\psi(s)\|_{\L^2}^2ds+C\beta(1+\gamma)\int_0^t\|v_{\psi}(s)\|_{\L^2}^2ds.
	\end{align} 
	An application of Gronwall's inequality in \eqref{7.32} gives 
	\begin{align}\label{7.33}
	&\|v_{\psi}(t)\|_{\L^2}^2+\nu\int_0^t\|\partial_xv_{\psi}(s)\|_{\L^2}^2ds+2\beta\gamma\int_0^t\|v_{\psi}(s)\|_{\L^2}^2ds\nonumber\\&\leq C\int_0^t\|\psi(s)\|_{\H_0^1}^2ds+\frac{\beta}{2}(1+\gamma)\int_0^t\|\psi(s)\|_{\L^2}^2ds\\&\quad\times\exp\left\{C\beta(1+\gamma)t+\int_0^t\left(\|v_{\psi_1}(s)\|_{\H_0^1}^2+\|\psi_1(s)\|_{\H_0^1}^2+\|v_{\psi_2}(s)\|_{\H_0^1}^2+\|\psi_2(s)\|_{\H_0^1}^2\right)ds\right\},\nonumber
	\end{align}
	for all $t\in[0,T]$. Let us now take $\psi_n\to\psi$ in $\C([0,T];\L^2(\mathcal{O}))\cap\L^2(0,T;\H_0^1(\mathcal{O})),$ as $n\to\infty$. From \eqref{7.17}, it is clear that 
	$$\sup_{t\in[0,T]}\|v_n(t)\|_{\L^2}^2+\nu\int_0^T\|v_n(t)\|_{\H_0^1}^2dt,$$ is bounded uniformly and independent of $n$. Thus, from \eqref{7.33}, it is immediate that $v_{\psi_n}\to v_{\psi}$ in $\C([0,T];\L^2(\mathcal{O}))\cap\L^2(0,T;\H_0^1(\mathcal{O}))$ and hence the continuity of the map $\Psi$ follows. 
\end{proof}

Let $\U_0:=\mathrm{Q}^{\frac{1}{2}}\L^2(\mathcal{O})$ and $\|\cdot\|_0$ denotes the norm on $\U_0$ and  $\mathcal{G}_0\left(\int_0^{\cdot}h(s)d s\right)$ is the set of solutions of the equation:
\begin{equation}\label{4.22}
\left\{
\begin{aligned}
dz(t)+ \nu Az(t)dt&=\sqrt{\e}h(t)dt, \ t\in(0,T),\\
z(0)&=0. 
\end{aligned}
\right.
\end{equation}

\begin{theorem}\label{thm4.8}
	Let $\Theta$ maps from $\mathscr{E}$ to $\mathscr{E}$ and is given by 
	\begin{align}
	\Theta(z)=z+\Psi(z),
	\end{align}
	where the map $\Psi(\cdot)$ is defined in \eqref{718} and \eqref{7.18}.
	For any given $R>0$ and $\delta >0$, there exists a large positive constant $\varrho_0$ such that for all $\varrho_0$, if we define the set $A_{\varrho}:=\Theta(\varrho\Theta^{-1}(B_R^c)),$ 
	then the unique pathwise strong solution $u(\cdot)$ of the system \eqref{abstract}  satisfies:
	\begin{align}\label{4.24}
	\mathbb{P}\left\{u\in A_{\varrho} \right\}\leq \exp\left\{-\varrho^2(\mathrm{J}(B_R^c)-\delta)\right\},
	\end{align}
	where  \begin{align*}B_R:=\bigg\{&v\in \mathrm{C}([0,T];\L^2(\mathcal{O})): \sup_{0\leq t\leq T}\|v(t)\|_{\L^2}^2<R \bigg\},\end{align*}
	\begin{align}
	\mathrm{J}(B_R^c)=\inf_{x\in\Theta^{-1}(B_R^c)}\mathrm{I}(x),
	\end{align}
	and 
	\begin{align}\label{4.26}
	\mathrm{I}(x)=\inf_{h\in\mathrm{L}^2(0,T;\U_0):\ x\in\mathcal{G}_0\left(\int_0^{\cdot}h(s)d s\right)}\left\{\frac{1}{2}\int_0^T\|h(t)\|_0^2d t \right\}.
	\end{align}
\end{theorem}

\begin{proof}
	For each $h\in\mathrm{L}^2(0,T;\U_0)$, we us use the notation $\mathcal{G}_0\left(\int_0^{\cdot}h(s)d s\right)$ for the set of solutions of the equation \eqref{4.22}. For each $\e>0$, let $z^{\e}(\cdot)$ denotes the unique pathwise strong solution of the stochastic heat equation: 
	\begin{equation}\label{4.23}
	\left\{
	\begin{aligned}
	dz^{\e}(t)+ \nu Az^{\e}(t)dt&=\sqrt{\e}dW(t), \ t\in(0,T),\\
	z^{\e}(0)&=0. 
	\end{aligned}
	\right.
	\end{equation}
	Then $z^{\e}(t)=\sqrt{\e}\int_0^tR{(t-s)}dW(s)=\sqrt{\e}z(t)$, where $R(\cdot)$ is the heat semigroup and $z(\cdot)$ is the unique pathwise strong solution of the system \eqref{7p1}. Note that (see section 12.3, \cite{DaZ}, \cite{SSSP}) the large deviations rate function for the family $z^{\e}$ is given by 
	\begin{align}
	\mathrm{I}(x)=\inf_{h\in\mathrm{L}^2(0,T;\U_0):\ x\in\mathcal{G}_0\left(\int_0^{\cdot}h(s)d s\right)}\left\{\frac{1}{2}\int_0^T\|h(t)\|_0^2d t \right\}.
	\end{align}
	Let us now define the map $\Theta$ from $\mathscr{E}$ to $\mathscr{E}$ by 
	$
	\Theta(z)=z+\Psi(z),
	$
	where the map $\Psi(\cdot)$ is defined in \eqref{718} and \eqref{7.18}. Clearly, the map $\Theta$ is continuous by using Lemma \ref{lem7.7} and \begin{align}\label{4.03}
	u^{\e}=\Theta(z^{\e})=\Theta(\sqrt{\e}z),\end{align} where $u^{\e}$ satisfies: 
	\begin{equation}\label{4.27}
	\left\{
	\begin{aligned}
	du^{\e}(t)&=[- Au^{\e}(t)-\alpha B(u^{\e}(t))+\beta c(u^{\e}(t))]dt+\sqrt{\e} dW(t), \ t\in(0,T),\\
	u(0)&=u_0\in\L^{2}(\mathcal{O}). 
	\end{aligned}
	\right.
	\end{equation}
	Then, using Contraction principle (see Theorem \ref{thm4.6}), we deduce that the family $u^{\e}$ satisfies the large deviation principle with the rate function: 
	\begin{align}
	\J(A)=\inf_{x\in\Theta^{-1}(A)}\I(x),
	\end{align}
	for any Borel set $A\in \mathscr{E}$, where $\Theta^{-1}(A)=\left\{x\in\mathscr{E}:\Theta(x)\in A\right\}.$ Thus, using the LDP (see Definition \ref{def4.2} (i)), we have  
	\begin{align}
	\limsup_{\e\to 0}\e\log \P\left\{u^{\e}\in B_R^c\right\}\leq -\J(B_R^c),
	\end{align}
	where $B_R$ is an open ball in $\mathscr{E}$ with center zero and radius $R>0$. Thus, for any $\delta>0$, there exists an $\varepsilon_1>0$ such that for all $0<\e<\e_1$, we have 
	\begin{align*}
	\e\log \P\left\{u^{\e}\in B_R^c\right\}\leq -\J(B_R^c)+\delta. 
	\end{align*}
	The above inequality easily gives
	\begin{align}\label{4.30}
	\P\left\{u^{\e}\in B_R^c\right\}\leq \exp\left\{-\frac{1}{\e}(\J(B_R^c)-\delta)\right\}.
	\end{align}
	From \eqref{4.30}, it is clear that 
	\begin{align}\label{4.31}
	\mathbb{P}\left\{z\in\frac{1}{\sqrt{\e}}\Theta^{-1}(B_R^c)\right\}\leq \exp\left\{-\frac{1}{\e}(\J(B_R^c)-\delta)\right\},
	\end{align}
	using \eqref{4.03}. Let us denote the set $A$ to be $\Theta\left(\frac{1}{\sqrt{\e}}\Theta^{-1}(B_R^c)\right)$ and from \eqref{4.31}, we infer that 
	\begin{align}\label{4.32}
	\mathbb{P}\left\{z\in A\right\}\leq \exp\left\{-\frac{1}{\e}(\J(B_R^c)-\delta)\right\},
	\end{align}
	since $u=\Theta(z)$, which completes the proof. 
\end{proof}
\begin{remark}
	If we take $\varrho_0=1$, then the set $A_1$ becomes $B_R^c$, and from \eqref{4.24}, we deuce that 
	\begin{align}\label{4.36}
	\mathbb{P}\left\{u\in B_R^c \right\}\leq \exp\left\{-\varrho^2(\mathrm{J}(B_R^c)-\delta)\right\},
	\end{align}
	which gives the	rate of decay as $\mathrm{J}(B_R^c)$. Moreover, if one can assure the existence of an $R >0$ such that $B_R^c\subseteq A_{\varrho_0}$, then also the Theorem \ref{thm4.8} leads to \eqref{4.36}. 
\end{remark}

\begin{remark}\label{rem5.10}
	An application of It\^o's formula to the process $\|u(\cdot)\|_{\L^2}^2$ yields 
		\begin{align}\label{5p50}
	&\|u(t)\|_{\L^2}^2+2\nu\int_0^t\|u(s)\|_{\H_0^1}^2d s+2\beta\gamma\int_0^t\|u(s)\|_{\L^2}^2ds+2\beta\int_0^t\|u(s)\|_{\L^4}^4ds\nonumber\\&=\|u_0\|_{\L^2}^2+2\beta(1+\gamma)\int_0^t(u(s)^2,u(s))ds+\int_0^t\Tr(\Phi Q\Phi^*) ds+2\int_0^t(\Phi d W(s),u(s))\nonumber\\&\leq \|u_0\|_{\L^2}^2+\beta\int_0^t\|u(s)\|_{\L^4}^4ds+\beta(1+\gamma)^2\int_0^t\|u(s)\|_{\L^2}^2ds\nonumber\\&\quad+\eta(t)+2\int_0^t\|Q^{1/2}u(s)\|_{\L^2}^2ds+t\Tr(Q),
	\end{align}
	where 
	\begin{align}\label{5p51}
	\eta(t)=2\int_0^t(d W(s),u(s))-2\int_0^t\|Q^{1/2}u(s)\|_{\L^2}^2ds.
	\end{align}
	But we know that 
	\begin{align}
	\|Q^{1/2}u\|_{\L^2}\leq\|Q^{1/2}\|_{\mathcal{L}(\L^2)}\|u\|_{\L^2}\leq\Tr(Q)^{1/2}\|u\|_{\L^2}.
	\end{align}
	Thus, from \eqref{5p51}, it is immediate that 
	\begin{align}
		&\|u(t)\|_{\L^2}^2+2\nu\int_0^t\|u(s)\|_{\H_0^1}^2d s+\beta\int_0^t\|u(s)\|_{\L^4}^4ds\nonumber\\&\leq \|u_0\|_{\L^2}^2+\Tr(Q) t+\eta(t)+\beta(1+\gamma^2)\int_0^t\|u(s)\|_{\L^2}^2ds+2\Tr(Q)\int_0^t\|u(s)\|_{\L^2}^2ds
	\end{align}
	An application of Gronwall's inequality yields 
	\begin{align}
	\|u(t)\|_{\L^2}^2&\leq \left(\|u_0\|_{\L^2}^2+\Tr(Q) t+\eta(t)\right)+\int_0^t \left(\|u_0\|_{\L^2}^2+\Tr(Q) s+\eta(s)\right)\left(\beta(1+\gamma^2)+2\Tr(Q)\right)\nonumber\\&\qquad\times\exp\left(\int_s^t\left(\beta(1+\gamma^2)+2\Tr(Q)\right)dr\right)ds.
	\end{align}
	Taking supremum over both sides in the above inequality, we find 
	\begin{align}
&	\sup_{t\in[0,T]}\|u(t)\|_{\L^2}^2\leq \left(\|u_0\|_{\L^2}^2+\Tr(Q) T+\sup_{t\in[0,T]}\eta(t)\right)M\exp\left(MT\right),
	\end{align}
where $M=	\beta(1+\gamma^2)+2\Tr(Q)$. For fixed $R>0$, we have 
	\begin{align}\label{5p56}
	\mathbb{P}\left\{	\sup_{t\in[0,T]}\|u(t)\|_{\L^2}>R\right\}& \leq \mathbb{P}\left\{\left(\|u_0\|_{\L^2}^2+\Tr(Q) T+\sup_{t\in[0,T]}\eta(t)\right)Me^{MT}>R^2\right\}\nonumber\\&=\mathbb{P}\left\{\sup_{t\in[0,T]}\eta(t)>\frac{R^2}{M}e^{-MT}-\left(\|u_0\|_{\L^2}^2+\Tr(Q)T\right)\right\}\nonumber\\&=\mathbb{P}\left\{\sup_{t\in[0,T]}\exp(\eta(t))>\exp\left[\frac{R^2}{M}e^{-MT}-\left(\|u_0\|_{\L^2}^2+\Tr(Q)T\right)\right]\right\}\nonumber\\&\leq \exp\left[\left(\|u_0\|_{\L^2}^2+\Tr(Q)T\right)\right]e^{-\frac{R^2}{M}e^{-MT}},
	\end{align}
	where we used Doob’s martingale inequality. 
	
	From the expression \eqref{5p56}, we know that the rate of decay is of the order of $R^2$. We can also follow the same procedure as in the Theorem \ref{thm4.8} to get a similar result. Let us define the set 
	\begin{align}
	\F_R:=\left\{x:\J(x)\leq R^2 \right\},
	\end{align}
	for $R> 0$ and define the set $\mathrm{G}_R$ as any open neighborhood of $\F_R$. Then for any given any $\delta>0$, there exists an $\e_1 > 0$ such that for all $0<\e < \e_1$, from \eqref{4.30}, we have
	\begin{align}
	\mathbb{P}\left\{u^{\e}\in\mathrm{G}_R^c\right\}\leq \exp\left\{-\frac{1}{\e}(\J(\mathrm{G}_R^c)-\delta)\right\}\leq \exp\left\{-\frac{1}{\e}(R^2-\delta)\right\},
	\end{align}
	using the definition of the set $\mathrm{G}_R$. Hence, using \eqref{4.03}, it is immediate that
	\begin{align}
	\mathbb{P}\left\{u\in\Psi\left(\frac{1}{\sqrt{\e}}\Psi^{-1}(\mathrm{G}_R^c)\right)\right\}\leq \exp\left\{-\frac{1}{\e}(R^2-\delta)\right\}.
	\end{align} 
\end{remark}

\section{Exponential Moments, Invariant Measures and Ergodicity}\label{sec9}\setcounter{equation}{0} 
Let us now discuss the existence and uniqueness of invariant measures and ergodicity results for the stochastic Burgers-Huxley equation \eqref{abstract} with additive Gaussian noise. We show that there exists a unique invariant measure for the Markovian transition probability associated to the  system (\ref{abstract}) by making use of exponential stability of solutions.

\subsection{Exponential moments and stability}
In this subsection, we establish the exponential stability of the Burgers-Huxley equation perturbed by additive Gaussian noise. That is, we are considering (\ref{2.1}) with additive noise. Thus $u(\cdot)$ satisfies: 
\begin{equation}\label{6.1a}
\left\{
\begin{aligned}
du(t)&=[- \nu Au(t)-\alpha B(u(t))+\beta c(u(t))]dt+dW(t), \ t\in(0,T),\\
u(0)&=u_0,
\end{aligned}
\right.
\end{equation}
where $u_0\in\L^2(\mathcal{O})$ and $\W(\cdot)$ is an $\L^2(\mathcal{O})$-valued $Q$-Wiener process with $\Tr(Q)<\infty$. Since $\Tr(Q)<\infty$, the existence and uniqueness of strong solution to the system (\ref{6.1a}) follows from the Theorem \ref{exis}. Thus, we know that the system \eqref{abstract} has a unique strong solution with paths in $\C([0,T];\L^2(\mathcal{O}))\cap\mathrm{L}^2(0,T;\H_0^1(\mathcal{O}))$, $\P$-a.s.  Then, we have the following Theorem on the exponential moments of the system \eqref{6.1a}. 
\begin{theorem}\label{expe}
	Let $u(\cdot)$ be a unique strong solution of the problem (\ref{6.1a}) such that  
	for \begin{align}\label{7.2}0<\e\leq \frac{\left(\nu\pi^2-{\beta(1+\gamma^2)}\right)}{2\|Q\|_{\mathcal{L}(\L^2)}}\ \text{ and }\  \nu > \frac{\beta(1+\gamma^2)}{\pi^2},\end{align} we have 
	\begin{align}\label{5.68}
	\E\left[\exp\left(\e\|u(t)\|_{\L^2}^2+\e\nu\int_0^t\|u(s)\|_{\H_0^1}^2d s+\e\beta\int_0^t\|u(s)\|_{\L^4}^4ds\right)\right]\leq e^{\e\|u_0\|_{\L^2}^2+\e t\Tr(Q)}.
	\end{align}
\end{theorem}
\begin{proof}
	Let us apply the infinite dimensional It\^o's formula to the process $\|u(\cdot)\|_{\L^2}^2$ to get 
	\begin{align}
	&\|u(t)\|_{\L^2}^2+2\nu\int_0^t\|u(s)\|_{\H_0^1}^2d s+2\beta\gamma\int_0^t\|u(s)\|_{\L^2}^2ds+2\beta\int_0^t\|u(s)\|_{\L^4}^4ds\nonumber\\&=\|u_0\|_{\L^2}^2+2\beta(1+\gamma)\int_0^t(u(s)^2,u(s))ds+\Tr(Q) t+2\int_0^t(d W(s),u(s)).
	\end{align}
	We define $$\Theta(t):=\|u(t)\|_{\L^2}^2+\nu\int_0^t\|u(s)\|_{\H_0^1}^2d s+\beta\int_0^t\|u(s)\|_{\L^4}^4ds,$$ and apply the infinite dimensional  It\^o's formula to the process $e^{\e\Theta(t)}$ to obtain 
	\begin{align}\label{5.62a}
	e^{\e\Theta(t)}&=e^{\e\|u_0\|_{\L^2}^2}+\e\int_0^te^{\e\Theta(s)}\left(-\nu\|u(s)\|_{\H_0^1}^2-\beta\|u(s)\|_{\L^4}^4-2\beta\gamma\|u(s)\|_{\L^2}^2+2\beta(1+\gamma)(u(s)^2,u(s))\right)d s\nonumber\\&\quad +\e\int_0^te^{\e\Theta(s)}\Tr(Q)d s+2\e\int_0^te^{\e\Theta(s)}\left(dW(s),u(s)\right) +2\e^2\int_0^te^{\e\Theta(s)}\|Q^{1/2}u(s)\|_{\L^2}^2d s,
	\end{align}
	since $\Tr((u\otimes u)Q)=\|Q^{1/2}u\|_{\L^2}^2$. Note that 
	\begin{align}\label{5.63}
2\beta(1+\gamma)	(u^2,u)\leq 2\beta(1+\gamma)	\|u\|_{\L^4}^2\|u\|_{\L^2}\leq\beta\|u\|_{\L^4}^4+\beta(1+\gamma)^2\|u\|_{\L^2}^2,
	\end{align}
	and 
	\begin{align}\label{5.64}
	\|Q^{1/2}u\|_{\L^2}^2\leq \|Q\|_{\mathcal{L}(\L^2)}\|u\|_{\L^2}^2\leq \frac{\|Q\|_{\mathcal{L}(\L^2)}}{\pi^2}\|u\|_{\H_0^1}^2.
	\end{align}
	Taking expectation in (\ref{5.62a}), and then using (\ref{5.63}) and (\ref{5.64}), we obtain 
	\begin{align}\label{5.66}
	\E\left[e^{\e\Theta(t)}\right]&\leq e^{\e\|u_0\|_{\L^2}^2}+\e\E\left\{\int_0^te^{\e\Theta(s)}\left[-\left(\nu-\frac{\beta(1+\gamma^2)}{\pi^2}\right)+\frac{2\e\|Q\|_{\mathcal{L}(\L^2)}}{\pi^2}\right]\|u(s)\|_{\H_0^1}^2d s\right\} \nonumber\\&\quad+\e\E\left[\int_0^te^{\e\Theta(s)}\Tr(Q)d s\right].
	\end{align}
	Now, for $0<\e\leq \frac{\left(\nu\pi^2-{\beta(1+\gamma^2)}\right)}{2\|Q\|_{\mathcal{L}(\L^2)}}\ \text{ and }\  \nu > \frac{\beta(1+\gamma^2)}{\pi^2}$, we have 
	\begin{align}\label{5.67}
	\E\left[e^{\e\Theta(t)}\right]&\leq e^{\e\|u_0\|_{\L^2}^2} +\e\int_0^t\E\left[e^{\e\Theta(s)}\right]\Tr(Q)d s.
	\end{align}
	An application of Gronwall's inequality in (\ref{5.67}) yields (\ref{5.68}).
\end{proof}
\begin{theorem}\label{exps}
	Let $u(\cdot)$ and $v(\cdot)$ be two solutions of the system (\ref{6.1a}) with the initial data $u_0,v_0\in\L^2(\mathcal{O})$, respectively. Then for \begin{align}\label{610} \nu>\frac{\beta(1+\gamma^2)}{\pi^2}\ \text{ and } \ \nu^3\pi^2-{\beta(1+\gamma^2)}\nu^2\geq 2C\alpha^2\Tr(Q), \end{align}we have 
	\begin{align}\label{5.68a}
	&\E\left[\|u(t)-v(t)\|_{\L^2}^2\right]\nonumber\\&\leq \|u_0-v_0\|_{\L^2}^2e^{\left\{{\frac{C\alpha^2}{\nu^2}\|u_0\|_{\L^2}^2}\right\}}\exp\left\{-\left[\left(\nu\pi^2-\frac{\beta(1+\gamma^2)}{2}\right)-\frac{C\alpha^2}{\nu^2}\Tr(Q)\right]t\right\}.
	\end{align}
\end{theorem}
\begin{proof}
	Let $w(t)=u(t)-v(t)$ and $w(\cdot)$ satisfies:
	\begin{equation}\label{7p12}
	\left\{
	\begin{aligned}
	dw(t)&=[- \nu Aw(t)-\alpha [B(u(t))-B(v(t))]+\beta [c(u(t))-c(v(t))]]dt, \ t\in(0,T),\\
	w(0)&=u_0-v_0.
	\end{aligned}
	\right.
	\end{equation}
	Taking inner product with $w(\cdot)$, it can be easily seen that 
	\begin{align}\label{7.13}
	\|w(t)\|_{\L^2}^2&=\|w_0\|_{\L^2}^2-2\nu\int_0^t\|w(s)\|_{\H_0^1}^2\d s -2\alpha\int_0^t\langle B(u(s))-B(v(s)),w(s)\rangle d s\nonumber\\&\quad+2\beta\int_0^t(c(u(s))-c(v(s)),w(s))ds\nonumber\\&=\|w_0\|_{\L^2}^2-2\nu\int_0^t\|w(s)\|_{\H_0^1}^2d s-2\alpha\int_0^t(u(s)\partial_xw(s),w(s))ds\nonumber\\&\quad-2\beta\gamma\int_0^t\|w(s)\|_{\L^2}^2ds-2\beta\int_0^t\|u(s)w(s)\|_{\L^2}^2ds-2\beta\int_0^t\|v(s)w(s)\|_{\L^2}^2ds\nonumber\\&\quad+\beta(1+\gamma)\int_0^t((u(s)+v(s))w(s),w(s))ds-\beta\int_0^t(u(s)v(s)w(s),w(s))ds.
	\end{align}
	We estimate $-2\alpha(u\partial_xw,w)$ using H\"older's and Young's inequalities as 
	\begin{align}\label{7.14}
	-2\alpha(u\partial_xw,w)\leq2\alpha\|u\|_{\L^{\infty}}\|w\|_{\H_0^1}\|w\|_{\L^2}\leq{\nu}\|w\|_{\H_0^1}^2+\frac{C\alpha^2}{\nu}\|u\|_{\H_0^1}^2\|w\|_{\L^2}^2,
	\end{align}
	where $C$ is the constant appearing in $\|u\|_{\L^{\infty}}\leq C\|u\|_{\H_0^1}$.  We estimate $\beta(1+\gamma)((u+v)w,w)$ and $-\beta(uvw,w)$ using H\"older's and Young's inequalities as 
	\begin{align}
	\beta(1+\gamma)((u+v)w,w)&\leq\beta(1+\gamma)(\|uw\|_{\L^2}+\|vw\|_{\L^2})\|w\|_{\L^2}\nonumber\\&\leq \frac{\beta}{2}\left(\|uw\|_{\L^2}^2+\|vw\|_{\L^2}^2\right)+\frac{\beta(1+\gamma)^2}{2}\|w\|_{\L^2}^2, \\-\beta(uvw,w)&\leq \frac{\beta}{2}\left(\|uw\|_{\L^2}^2+\|vw\|_{\L^2}^2\right). \label{7p16}
	\end{align}
	Combining \eqref{7.14}-\eqref{7p16} and substituting it in \eqref{7.13}, we get 
	\begin{align}
		\|w(t)\|_{\L^2}^2&\leq \|w_0\|_{\L^2}^2-\int_0^t\left(-\nu\pi^2+\frac{\beta(1+\gamma^2)}{2}\right)\|w(s)\|_{\L^2}^2d s+\frac{C\alpha^2}{\nu}\int_0^t\|u(s)\|_{\H_0^1}^2\|w(s)\|_{\L^2}^2ds\nonumber\\&\quad-\beta\gamma\int_0^t\|w(s)\|_{\L^2}^2ds-\beta\int_0^t\|u(s)w(s)\|_{\L^2}^2ds-\beta\int_0^t\|v(s)w(s)\|_{\L^2}^2ds\nonumber\\&\leq \|w_0\|_{\L^2}^2-\int_0^t\left(-\nu\pi^2+\frac{\beta(1+\gamma^2)}{2}\right)\|w(s)\|_{\L^2}^2d s+\frac{C\alpha^2}{\nu}\int_0^t\|u(s)\|_{\H_0^1}^2\|w(s)\|_{\L^2}^2ds.
	\end{align}
	Thus, by applying Gronwall's inequality, we obtain 
	\begin{align}
	\|w(t)\|_{\L^2}^2\leq \|w_0\|_{\L^2}^2\exp\left\{-\left(\nu\pi^2-\frac{\beta(1+\gamma^2)}{2}\right)t+\frac{C\alpha^2}{\nu}\int_0^t\|u(s)\|_{\H_0^1}^2d s\right\}.
	\end{align}
	and 
	\begin{align}
	\E\left[\|w(t)\|_{\L^2}^2\right]&\leq\|w_0\|_{\L^2}^2e^{-\left(\nu\pi^2-\frac{\beta(1+\gamma^2)}{2}\right)t}\E\left[\exp\left(\frac{C\alpha^2}{\nu}\int_0^t\|u(s)\|_{\H_0^1}^2d s\right)\right]\nonumber\\& \leq \|w_0\|_{\L^2}^2e^{-\left\{\left(\nu\pi^2-\frac{\beta(1+\gamma^2)}{2}\right)-\frac{C\alpha^2}{\nu^2}\Tr(Q)\right\}t}\exp\left\{{\frac{C\alpha^2}{\nu^2}\|u_0\|_{\L^2}^2}\right\},
	\end{align}
	where we used the bound given in \eqref{5.68} for $\nu^3\pi^2-\beta(1+\gamma^2)\nu^2\geq 2C\alpha^2\|Q\|_{\mathcal{L}(\L^2)}$. Thus,   for $\nu>\frac{\beta(1+\gamma^2)}{2\pi^2}$ and $\nu^3\pi^2-\frac{\beta(1+\gamma^2)}{2}\nu^2\geq 2C\alpha^2\Tr(Q)$, we get the required result given in \eqref{5.68a}. Since $\|Q\|_{\mathcal{L}(\L^2)}\leq\Tr(Q)$, the condition given in \eqref{610} is a sufficient condition for obtaining the estimate \eqref{5.68a}.
\end{proof}

\subsection{Preliminaries}
In this subsection, we give the definitions of invariant measures,  ergodic, strongly mixing and exponentially mixing invariant measures. 

Let $\mathscr{E}$ be a Polish space.  
\begin{definition}
	A probability measure $\mu$ on
	$(\mathscr{E},\mathscr{B}(\mathscr{E}))$ is called \emph{an invariant
		measure or a stationary measure} for a given transition probability
	function $P(t,x,d y),$ if it satisfies
	$$\mu(A)=\int_{\mathscr{E}}{P}(t,x,A)d\mu(x),$$ for all $A\in\mathscr{B}(\mathscr{E})$ and
	$t>0$. Equivalently, if for all $\varphi\in \mathrm{C}_b(\mathscr{E})$
	(the space of bounded continuous functions on $\mathscr{E}$), and all
	$t\geq 0$,
	$$\int_{\mathscr{E}}\varphi(x)d\mu(x)=\int_{\mathscr{E}}(P_t\varphi)(x)d\mu(x),$$ where the Markov semigroup
	$(P_t)_{t\geq 0}$ is defined by
	$$P_t\varphi(x)=\int_{\mathscr{E}}\varphi(y)P(t,x,d y).$$
\end{definition}
\begin{definition}[Theorem 3.2.4, Theorem 3.4.2, \cite{GDJZ}]
	Let $\mu$ be an invariant measure for $\left(P_t\right)_{t\geq 0}.$  We say that the measure $\mu$ is an \emph{ergodic measure,}  if for all $\varphi \in \L^2(\mathscr{E};\mu), $ we have  $$ \lim_{T\to +\infty}\frac{1}{T}\int_0^T (P_t\varphi)(x) d t =\int_{\mathscr{E}}\varphi(x) d\mu(x) \ \text{ in } \ \L^2(\mathscr{E};\mu).$$ The invariant measure $\mu$ for $\left(P_t\right)_{t\geq 0}$ is called \emph{strongly mixing} if  for all $\varphi \in \L^2(\mathscr{E};\mu),$  we have $$\lim_{t\to+\infty}P_t\varphi(x) = \int_{\mathscr{E}}\varphi(x) d\mu(x)\ \text{ in }\ \L^2(\mathscr{E};\mu).$$ The invariant measure $\mu$ for $\left(P_t\right)_{t\geq 0}$ is called \emph{exponentially mixing}, if there exists a constant $k>0$ and a positive function $\Psi(\cdot)$ such that for any bounded Lipschitz function $\varphi$, all $t>0$ and all $x\in\mathscr{E}$, $$\left|P_t\varphi(x)-\int_{\mathscr{E}}\varphi(x)d\mu(x)\right|\leq \Psi(x)e^{-k t}\|\varphi\|_{\text{Lip}},$$  where $\|\cdot\|_{\text{Lip}}$ is the Lipschitz constant. 
\end{definition}
\begin{remark}
	Clearly exponentially mixing implies strongly mixing. Theorem 3.2.6, \cite{GDJZ} states that if  $\mu$ is the unique invariant measure for $(P_t)_{t\geq 0}$, then  it is ergodic. 
\end{remark}
The interested readers are referred to see \cite{GDJZ} for more details on the ergodicity for infinite dimensional systems. 
\subsection{Existence of a unique invariant measure} In this subsection, we show that there exists a unique invariant measure for the Markovian transition probability associated to the  system (\ref{6.1a}). Moreover, we show that the invariant measure is ergodic and strongly mixing (in fact exponentially mixing). Let $u(t;u_0)$ denotes the unique strong solution of the system (\ref{6.1a}) with the initial condition
$u_0\in\L^2(\mathcal{O}).$ Let $(P_t)_{t\geq 0}$ be the \emph{Markovian transition semigroup} in
the space $\C_b(\L^2(\mathcal{O}))$ associated to the system (\ref{6.1a}) defined
by
\begin{align}\label{mar}
P_t\varphi(u_0)=\E\left[\varphi(u(t;u_0))\right]=\int_{\L^2}\varphi(y)P(t,u_0,\d
y)=\int_{\L^2}\varphi(y)\mu_{t,u_0}(d y),\;\varphi\in \C_b(\L^2(\mathcal{O})),
\end{align}
where $P(t,u_0,d y)$ is the transition probability of
$u(t;u_0)$ and $\mu_{t,u_0}$ is the law of $u(t;u_0)$. The semigroup $(P_t)_{t\geq 0}$ is Feller, since the solution to
\eqref{6.1a} depends continuously on the initial data. From
(\ref{mar}), we also have
\begin{align}\label{amr}
P_t\varphi(u_0)=\left<\varphi,\mu_{t,u_0}\right>=\left<P_t\varphi,\mu\right>,
\end{align}
where $\mu$ is the law of the initial data $u_0\in\L^2(\mathcal{O})$. Thus from
(\ref{amr}), we have $\mu_{t,u_0}=P_t^*\mu$. We say that a
probability measure $\mu$ on $\L^2(\mathcal{O})$ is an \emph{invariant measure} if
\begin{align}
P_t^*\mu=\mu,\textrm{ for all }\ t\geq 0.
\end{align}
That is, if a solution has law $\mu$ at some time, then it has the same law for all later times. For such a solution, it can be shown by Markov property that for all $(t_1,\ldots,t_n)$ and $\tau>0$, $(u(t_1+\tau;u_0),\ldots,u(t_n+\tau;u_0))$ and $(u(t_1;u_0),\ldots,u(t_n;u_0))$ have the same law. Then, we say that the process $u$ is \emph{stationary}. For more details, the interested readers are referred to see \cite{GDJZ,ADe}, etc.
\begin{theorem}\label{EIM}
	Let $u_0\in\L^2(\mathcal{O})$ be given and $\Tr(Q)<+\infty$. Then, for $\nu>\frac{\beta(1+\gamma^2)}{2\pi^2},$ 
	there exists an invariant measure for the system (\ref{6.1a}) with support in $\H_0^1(\mathcal{O})$.
\end{theorem}
\begin{proof}
	Let us apply the infinite dimensional It\^o's formula to the process $\|u(\cdot)\|_{\L^2}^2$ to get
	\begin{align}\label{7p4}
	&\|u(t)\|_{\L^2}^2+2\nu\int_0^t\|u(s)\|_{\H_0^1}^2d s+2\beta\gamma\int_0^t\|u(s)\|_{\L^2}^2ds+2\beta\int_0^t\|u(s)\|_{\L^4}^4ds\nonumber\\&=\|u_0\|_{\L^2}^2+2\beta(1+\gamma)\int_0^t(u(s)^2,u(s))ds+\Tr(Q) t+2\int_0^t(d W(s),u(s))\nonumber\\&\leq \|u_0\|_{\L^2}^2+\beta\int_0^t\|u(s)\|_{\L^4}^4ds+\beta(1+\gamma)^2\int_0^t\|u(s)\|_{\L^2}^2ds+\Tr(Q) t+2\int_0^t(d W(s),u(s)).
	\end{align}
 Taking expectation in (\ref{7p4}), using Poincar\'e inequality, (\ref{2.10}) and the fact that the final term is a martingale having zero expectation, we obtain  
	\begin{align}\label{5.4}
	&	\E\left\{\|u(t)\|_{\L^2}^2+2\left(\nu-\frac{\beta(1+\gamma^2)}{2\pi^2}\right)\int_0^t\|u(s)\|_{\H_0^1}^2d s+2\beta\int_0^t\|u(s)\|_{\L^4}^4ds\right\}\nonumber\\&\leq
	\E\left[\|u_0\|_{\L^2}^2\right]+\Tr(Q)t.
	\end{align}
	Hence, for $\nu>\frac{\beta(1+\gamma^2)}{2\pi^2},$, we have 
	\begin{align}\label{5.6}
	\frac{\xi}{t}\E\left[\int_0^{t}\|u(s)\|_{\H_0^1}^2d s\right]\leq
	\frac{1}{T_0}\|u_0\|_{\L^2}^2+\Tr(Q),\text{ for all }t>T_0,
	\end{align}
	where $\xi=2\left(\nu-\frac{\beta(1+\gamma^2)}{2\pi^2}\right)$. Thus using Markov's inequality, we have
	\begin{align}\label{5.7}
	&\lim_{R\to\infty}\sup_{T>T_0}\left[\frac{1}{T}\int_0^T\P\Big\{\|u(t)\|_{\H_0^1}>R\Big\}d
	t\right]\nonumber \\&\leq
	\lim_{R\to\infty}\sup_{T>T_0}\frac{1}{R^2}\E\left[\frac{1}{T}\int_0^T\|u(t)\|_{\H_0^1}^2d
	t\right]=0.
	\end{align}
	Hence along with the estimate in (\ref{5.7}),  using the compactness of $\H_0^1(\mathcal{O})$ in $\L^2(\mathcal{O})$, it is clear by a standard argument that the sequence of probability measures $$\mu_{t,u_0}(\cdot)=\frac{1}{t}\int_0^t\Pi_{s,u_0}(\cdot)d s,\ \text{ where }\ \Pi_{t,u_0}(\Lambda)=\mathbb{P}\left(u(t;u_0)\in\Lambda\right), \ \Lambda\in\mathscr{B}(\L^2(\mathcal{O})),$$ is tight, that is, for each $\e>0$, there is a compact subset $K\subset\L^2(\mathcal{O})$  such that $\mu_t(K^c)\leq \e$, for all $t>0$, and so by the Krylov-Bogoliubov theorem (or by a  result of Chow and Khasminskii see \cite{CHKH}) $\mu_{t_n,u_0}\to\mu$ weakly for $n\to\infty$, and $\mu$ results to be an invariant measure for the transition semigroup $(P_t)_{t\geq 0}$,  defined by 	$$P_t\varphi(u_0)=\E\left[\varphi(u(t;u_0))\right],$$ for all $\varphi\in\C_b(\L^2(\mathcal{O}))$, where $u(\cdot)$ is the unique strong solution of (\ref{6.1a}) with initial condition $u_0\in\L^2(\mathcal{O})$.
\end{proof}

Now we establish the uniqueness of invariant measure for the system (\ref{6.1a}) using the exponential stability results established in  Theorem \ref{exps}.  Similar results for 2D stochastic Navier-Stokes equations is established in \cite{ADe}, 2D magnetohydrodynamics systems is obtained in \cite{UMTM} and 2D Oldroyd model of  order one is proved in \cite{MTM3}.

\begin{theorem}\label{UEIM}
	Let the conditions given in Theorems \ref{expe} and  \ref{exps} hold true and $u_0\in\L^2(\mathcal{O})$ be given. Then, for  the condition given in \eqref{5.68a}, there is a unique invariant measure $\mu$ to the system (\ref{6.1a}). The measure $\mu$ is ergodic and strongly mixing, that is, \begin{align}\label{6.9a}\lim_{t\to\infty}P_t\varphi(u_0)=\int_{\L^2}\varphi(v_0)d\mu(v_0), \ \mu\text{-a.s., for all }\ u_0\in\L^2(\mathcal{O})\ \text{ and }\  \varphi\in\C_b(\L^2(\mathcal{O})).\end{align} Moreover, we have 
	\begin{align}\label{7.25}
	\int_{\L^2}\exp\left(\e\|u_0\|_{\L^2}^2\right)d\mu(u_0)<+\infty,
	\end{align}
	where 
	\begin{align}\label{721}
	0<\e\leq \frac{\left(\nu\pi^2-{\beta(1+\gamma^2)}\right)}{2\|Q\|_{\mathcal{L}(\L^2)}}\ \text{ and }\  \nu > \frac{\beta(1+\gamma^2)}{\pi^2}. 
	\end{align}
\end{theorem}
\begin{proof}
	\textbf{Step (1):} \emph{Uniqueness of invariant measure $\mu$.} 	For $\varphi\in \text{Lip}(\L^2(\mathcal{O}))$ (Lipschitz $\varphi$), since $\mu$ is an invariant measure, we have 
	\begin{align}
	&	\left|P_t\varphi(u_0)-\int_{\L^2}\varphi(v_0)\mu(d v_0)\right|\nonumber\\&=	\left|\E[\varphi(u(t;u_0))]-\int_{\L^2}P_t\varphi(v_0)\mu(d v_0)\right|\nonumber\\&=\left|\int_{\L^2}\E\left[\varphi(u(t;u_0))-\varphi(u(t;v_0))\right]\mu(d v_0)\right|\nonumber\\&\leq L_{\varphi}\int_{\L^2}\E\left\|u(t;u_0)-u(t;v_0)\right\|_{\L^2}\mu(d v_0)\nonumber\\&\leq L_{\varphi}\exp{\left\{{\frac{C\alpha^2}{\nu^2}\|u_0\|_{\L^2}^2}\right\}}e^{-\widehat{\kappa}t}\int_{\L^2}\|u_0-v_0\|_{\L^2}\mu(d v_0)\nonumber\\&\leq L_{\varphi}\exp{\left\{{\frac{C\alpha^2}{\nu^2}\|u_0\|_{\L^2}^2}\right\}}e^{-\widehat{\kappa}t}\left(\|u_0\|_{\L^2}+\int_{\L^2}\|v_0\|_{\L^2}\mu(d v_0)\right)\nonumber\\&\to 0\text{ as } t\to\infty,
	\end{align}
	since $\int_{\L^2}\|v_0\|_{\L^2}\mu(d v_0)<+\infty$, where $\widehat{\kappa}=\left(\nu\pi^2-\frac{\beta(1+\gamma^2)}{2}\right)-\frac{C\alpha^2}{\nu^2}\Tr(Q)>0$. Hence, we deduce (\ref{6.9a}), for every $\varphi\in \C_b (\L^2(\mathcal{O}))$, by the density of $\text{Lip}(\L^2(\mathcal{O}))$ in $\C_b (\L^2(\mathcal{O}))$. Note that, we have a stronger result that $P_t\varphi(u_0)$ converges exponentially fast to equilibrium, which is the exponential mixing property. This easily gives uniqueness of the invariant measure also. Indeed if  $\wi\mu$ is an another invariant measure, then
	\begin{align}
	&	\left|\int_{\L^2}\varphi(u_0)\mu(d u_0)-\int_{\L^2}\varphi(v_0)\wi\mu(d v_0)\right|\nonumber\\&= \left|\int_{\L^2}P_t\varphi(u_0)\mu(d u_0)-\int_{\L^2}P_t\varphi(v_0)\wi\mu(d v_0)\right|\nonumber\\&=\left|\int_{\L^2}\int_{\L^2}\left[P_t\varphi(u_0)-P_t\varphi(v_0)\right]\mu(d u_0)\wi\mu(d v_0)\right|\nonumber\\&\leq L_{\varphi}\exp{\left\{{\frac{C\alpha^2}{\nu^2}\|u_0\|_{\L^2}^2}\right\}}e^{-\widehat{\kappa}t}\int_{\L^2}\int_{\L^2}\|u_0-v_0\|_{\L^2}\mu(d u_0)\wi\mu(d v_0)\nonumber\\&\to 0\ \text{ as }\  t\to\infty,
	\end{align}
and uniqueness follows.	By Theorem 3.2.6, \cite{GDJZ}, since $\mu$ is the unique invariant measure for $(P_t)_{t\geq 0}$, we know that it is ergodic. 

\vskip 0.2cm
	
	\noindent \textbf{Step (2):} \emph{Proof of \eqref{7.25}.}	In order to prove \eqref{7.25}, we use a stationary solution $u(\cdot)$ with invariant law $\mu$. Note that the process $\|u(\cdot)\|_{\L^2}^2$ satisfies:
	\begin{align}\label{722}
	&\|u(t)\|_{\L^2}^2+2\nu\int_0^t\|u(s)\|_{\H_0^1}^2d s+2\beta\gamma\int_0^t\|u(s)\|_{\L^2}^2ds+2\beta\int_0^t\|u(s)\|_{\L^4}^4ds\nonumber\\&=\|u_0\|_{\L^2}^2+2\beta(1+\gamma)\int_0^t(u(s)^2,u(s))ds+\Tr(Q) t+2\int_0^t(d W(s),u(s)).
	\end{align}
	Let us now apply the infinite dimensional It\^o's formula to the process $\exp\left(\e\|u(t)\|_{\L^2}^2\right)$  to obtain
	\begin{align}\label{7.28}
	&\exp\left(\e\|u(t)\|_{\L^2}^2\right)+2\e\nu\int_0^t\exp\left(\e\|u(s)\|_{\L^2}^2\right)\|u(s)\|_{\H_0^1}^2d s\nonumber\\&\quad+2\e\beta\gamma\int_0^t\exp\left(\e\|u(s)\|_{\L^2}^2\right)\|u(s)\|_{\L^2}^2ds+2\e\beta\int_0^t\exp\left(\e\|u(s)\|_{\L^2}^2\right)\|u(s)\|_{\L^4}^4ds\nonumber\\&=\exp\left(\e\|u_0\|_{\L^2}^2\right)+2\e\beta(1+\gamma)\int_0^t\exp\left(\e\|u(s)\|_{\L^2}^2\right)(u(s)^2,u(s))ds\nonumber\\&\quad +\e\int_0^t\exp\left(\e\|u(s)\|_{\L^2}^2\right)\Tr(Q)d s+2\e\int_0^t\exp\left(\e\|u(s)\|_{\L^2}^2\right)\left(dW(s),u(s)\right) \nonumber\\&\quad+2\e^2\int_0^t\exp\left(\e\|u(s)\|_{\L^2}^2\right)\|Q^{1/2}u(s)\|_{\L^2}^2d s.
	\end{align} 
	Let  us take expectation in (\ref{7.28}) and use and an estimate similar to (\ref{5.63}) to obtain 
	\begin{align}\label{7.29}
		&\E\left[\exp\left(\e\|u(t)\|_{\L^2}^2\right)\right]+2\e\left(\nu-\frac{\beta(1+\gamma^2)}{2\pi^2}\right)\E\left[\int_0^t\exp\left(\e\|u(s)\|_{\L^2}^2\right)\|u(s)\|_{\H_0^1}^2d s\right]\nonumber\\&\quad+\e\beta\E\left[\int_0^t\exp\left(\e\|u(s)\|_{\L^2}^2\right)\|u(s)\|_{\L^4}^4ds\right]\nonumber\\&\leq\E\left[\exp\left(\e\|u_0\|_{\L^2}^2\right)\right] +\e\E\left[\int_0^t\exp\left(\e\|u(s)\|_{\L^2}^2\right)\Tr(Q)d s\right] \nonumber\\&\quad+2\e^2\E\left[\int_0^t\exp\left(\e\|u(s)\|_{\L^2}^2\right)\|Q^{1/2}u(s)\|_{\L^2}^2d s\right].
	\end{align}
	From (\ref{5.64}), we know that $\|Q^{1/2}u\|_{\L^2}^2\leq \frac{\|Q\|_{\mathcal{L}(\L^2)}}{\pi^2}\|u\|_{\H_0^1}^2$ and hence choosing $\e>0$ given in \eqref{721} so that 
	$$2\e\|Q^{1/2}u\|_{\L^2}^2\leq \left(\nu-\frac{\beta(1+\gamma^2)}{2\pi^2}\right)\|u\|_{\H_0^1}^2,\ \text{ for } \ \nu>\frac{\beta(1+\gamma^2)}{2\pi^2},$$
	and thus  from (\ref{7.29}), we obtain 
	\begin{align}\label{7.30}
	&\E\left[\exp\left(\e\|u(t)\|_{\L^2}^2\right)\right]+\e\left(\nu-\frac{\beta(1+\gamma^2)}{2\pi^2}\right)\E\left[\int_0^t\exp\left(\e\|u(s)\|_{\L^2}^2\right)\|u(s)\|_{\H_0^1}^2d s\right] \nonumber\\&\leq\exp\left(\e\|u_0\|_{\L^2}^2\right)  +\e\Tr(Q)\E\left[\int_0^t\exp\left(\e\|u(s)\|_{\L^2}^2\right)d s\right].
	\end{align}
	We know that $u(\cdot)$ is stationary with law $\mu$, so that we have 
	\begin{align}\label{7p31}
	\E\left[\exp\left(\e\|u(t)\|_{\L^2}^2\right)\right]=\E\left[\exp\left(\e\|u_0\|_{\L^2}^2\right)\right]=\int_{\L^2}\exp\left(\e\|u_0\|_{\L^2}^2\right)d\mu(u_0),
	\end{align}
	and 
	\begin{align}\label{7p32}
	\E\left[\int_0^t\exp\left(\e\|u(s)\|_{\L^2}^2\right)\|u(s)\|_{\H_0^1}^2d s\right]=t\int_{\L^2}\exp\left(\e\|u_0\|_{\L^2}^2\right)\|u_0\|_{\H_0^1}^2d\mu(u_0).
	\end{align}
	Using (\ref{7p31}) and (\ref{7p32}) in (\ref{7.30}), we obtain 
	\begin{align}\label{7p33}
	\int_{\L^2}\exp\left(\e\|u_0\|_{\L^2}^2\right)\|u_0\|_{\H_0^1}^2d\mu(u_0)\leq\frac{\Tr(Q)}{\left(\nu-\frac{\beta(1+\gamma^2)}{2\pi^2}\right)}\int_{\L^2}\exp\left(\e\|u_0\|_{\L^2}^2\right)d\mu(u_0).
	\end{align}
	Now for $R>0$, using the Poincar\'e inequality and (\ref{7p33}), we have (see \cite{ADe})
	\begin{align}
	\int_{\L^2}\exp\left(\e\|u_0\|_{\L^2}^2\right)d\mu(u_0)&=\int_{\|u_0\|_{\L^2}\leq R}\exp\left(\e\|u_0\|_{\L^2}^2\right)d\mu(u_0)+\int_{\|u_0\|_{\L^2}> R}\exp\left(\e\|u_0\|_{\L^2}^2\right)d\mu(u_0)\nonumber\\&\leq \exp\left(\e R^2\right)+\frac{1}{R^2}\int_{\L^2}\exp\left(\e\|u_0\|_{\L^2}^2\right)\|u_0\|_{\L^2}^2d\mu(u_0)\nonumber\\&\leq \exp\left(\e R^2\right)+\frac{1}{R^2\pi^2}\int_{\L^2}\exp\left(\e\|u_0\|_{\L^2}^2\right)\|u_0\|_{\H_0^1}^2d\mu(u_0)\nonumber\\&\leq \exp\left(\e R^2\right)+\frac{\Tr(Q)}{R^2\pi^2\left(\nu-\frac{\beta(1+\gamma^2)}{2\pi^2}\right)}\int_{\L^2}\exp\left(\e\|u_0\|_{\L^2}^2\right)d\mu(u_0).
	\end{align}
	Let us take $$\frac{\Tr(Q)}{R^2\pi^2\left(\nu-\frac{\beta(1+\gamma^2)}{2\pi^2}\right)}=\frac{1}{2},$$ to obtain (\ref{7.25}), which completes the proof. 
\end{proof}

\medskip\noindent
{\bf Acknowledgments:} M. T. Mohan would  like to thank the Department of Science and Technology (DST), India for Innovation in Science Pursuit for Inspired Research (INSPIRE) Faculty Award (IFA17-MA110) and Indian Institute of Technology Roorkee, for providing stimulating scientific environment and resources.

\end{document}